\documentclass[11pt,reqno]{amsart}
\usepackage{graphicx, fullpage,bbm, subfigure}
\usepackage{amssymb, amsmath, amsthm,tikz, bm,bbm}
\usepackage{epstopdf,enumitem,esint}
\usepackage{verbatim,mathabx}
\usepackage{xcolor,mathrsfs,mathtools,stmaryrd,microtype}
\usepackage[T1]{fontenc}
\usepackage[utf8]{inputenc}
\usepackage{hyperref}
\newcommand{\be}{\begin{equation} }
\newcommand{\ee}{\end{equation}}
\newcommand{\bse}{\begin{subequations}}
\newcommand{\ese}{\end{subequations}}

\newcommand{\LB}{\left[}
\newcommand{\RB}{\right]}
\newcommand{\LC}{\left(}
\newcommand{\RC}{\right)}

\newcommand{\mr}{\mathbb{R}}

\newcommand{\p}{\partial}

\newcommand{\R}{{\mathbb R}}

\theoremstyle{plain}
\newtheorem{theorem}{Theorem}[section]

\newtheorem{lemma}{Lemma}[section]
\newtheorem{proposition}{Proposition}[section]
\newtheorem{definition}{Definition}[section]
\theoremstyle{remark}
\newtheorem{remark}{Remark}[section]

\theoremstyle{definition}

\numberwithin{equation}{section}

\title[Transverse spectral stability of $b$-KP waves]{Spectral analysis of periodic $b$-KP equation under transverse perturbation}

\author[R. M. Chen]{Robin Ming Chen}
\address[Robin Ming Chen]{Department of Mathematics, University of Pittsburgh, Pittsburgh, PA 15260}
\email{mingchen@pitt.edu}
\author[L. Fan]{Lili Fan}%
\address[Lili Fan]{College of Mathematics and Information Science,
Henan Normal University, Xinxiang 453007, China}
\email{fanlily89@126.com}
\author[X. Wang]{Xingchang Wang}
\address[Xingchang Wang]{College of Mathematical Sciences, Harbin Engineering University, Harbin 150001, China}
\email{wangxingchang@hrbeu.edu.cn}
\author[R. Xu]{Runzhang Xu}
\address[Runzhang Xu]{College of Mathematical Sciences, Harbin Engineering University, Harbin 150001, China}
\email{xurunzh@hrbeu.edu.cn}

\date{}

\begin{document}

\begin{abstract}

The $b$-family-Kadomtsev-Petviashvili equation ($b$-KP) is a two dimensional generalization of the $b$-family equation. In this paper, we study the spectral stability of the one-dimensional small-amplitude periodic traveling waves with respect to two-dimensional perturbations which are either co-periodic in the direction of propagation, or nonperiodic (localized or bounded). We perform a detailed spectral analysis of the linearized problem associated to the above mentioned perturbations, and derive various stability and instability criteria which depends in a delicate way on the parameter value of $b$, the transverse dispersion parameter $\sigma$, and the wave number $k$ of the longitudinal waves.

%For the so-called CH-KP-I equation, we show that these periodic waves are unstable with respect to both types of perturbations and for the CH-KP-II equation, we show that these periodic waves are spectrally stable with respect to perturbations that are periodic in the direction of propagation, and have long wavelengths in the transverse direction.
\end{abstract}

\thispagestyle{empty}
\maketitle

\setcounter{tocdepth}{1}
\tableofcontents


%\vskip 0.2cm


\section{Introduction}

In the study of wave phenomena, particularly in dispersive and nonlinear media, the emergence of periodic wave trains stands out as a captivating outcome arising from the intricate interplay between dispersion and nonlinearity.  These wave patterns manifest across diverse physical domains, encompassing phenomena such as water waves, nonlinear optics, acoustics, and plasma. Given their pervasive presence in nature, the investigation of periodic wave trains continues to attract considerable interest from scientists and researchers.

One important question is the stability of the periodic waves. Stability properties govern the long-term behavior of the wave patterns and play a crucial role in understanding the robustness and predictability of the phenomena they represent. While many physical systems support unidirectional wave propagation, which can be modeled by equations in a single spatial dimension, it is crucial to acknowledge that in a multi-dimensional context, transverse effects become integral. Consequently, the stability analysis of one-dimensional traveling waves, particularly concerning perturbations propagating along the transverse direction of the primary axis -- referred to as transverse stability -- becomes a naturally compelling avenue of interest. This exploration extends beyond the traditional analysis of responses to perturbations along the main propagation axis, providing a more comprehensive understanding of the intricate dynamics governing wave stability in multi-dimensional settings.

Such a problem was first studied by Kadomtsev and Petviashvili \cite{KP1970}, who derived a two-dimensional generalization of the celebrated KdV equation, the so-called Kadomtsev-Petviashvili (KP) equation. They found that the KdV localized solitons in the KP flow are stable to transverse perturbations in the case of negative dispersion (KP-II), while unstable by long wavelength transverse perturbations for the positive dispersion model (KP-I). Later development of the theory for solitary waves includes the use of integrability \cite{Zak75}, explicit spectral analysis \cite{APS}, perturbation analysis \cite{KTN}, general Hamiltonian PDE techniques \cite{RT1,RT2}, Miura transformation \cite{MT2012}, the combination of algebraic properties, weighted function spaces, and refined PDE tools \cite{Miz2015,Miz2018}, among others.

When periodic waves are considered, to the authors' knowledge, most of the study of transverse stability pertains to spectral analysis; see for e.g., \cite{Spe1988,JZ10,Har11,HW,BKP} for the KP and generalized KP equations, \cite{HSS12,AC21} for the nonlinear Schr\"odinger (NLS) equation, and \cite{CW12,Joh10,Nat23} for the Zakharov-Kuznetsov (ZK) equation.

The goal of this paper is to extend the transverse stability analysis to periodic waves arising from models exhibiting strong non-local and nonlinear features. Specifically, we choose the one-dimensional $b$-family equation \cite{HS03}
\begin{equation}\label{eq b-family}
\left(1-\partial_x^2\right) u_t+(b+1) u u_x+ \kappa u_x-b u_x u_{x x}-u u_{x x x} = 0, \quad u = u(t,x), \ \ b \in \R, \kappa > 0,
\end{equation}
and consider its two-dimensional generalization
\begin{equation}\label{CHKP}
\left[\left(1-\partial_x^2\right) u_t+(b+1) u u_x+ \kappa u_x-b u_x u_{x x}-u u_{x x x}\right]_x+\sigma u_{y y}=0, \quad \sigma = \pm1,
\end{equation}
where the profile $u = u(t,x,y)$. We refer to \eqref{CHKP} as the $b$-KP equation due to the resemblance of transverse term $\sigma u_{yy}$ to that of the classical KP equation; and in a similar way, the $b$-KP equation with $\sigma = -1$ is called the $b$-KP-I equation, whereas the one with $\sigma=1$ is called the $b$-KP-II equation. The physical relevance of the $b$-KP equation \eqref{CHKP} has been recently discovered in the context of shallow water waves \cite{Jo02,GLL21} and nonlinear elasticity \cite{Ch06}, for the case $b = 2$. The corresponding equation is also referred to as the CH-KP equation as it generalizes the well-known Camassa--Holm equation \cite{CH}.

While the (longitudinal) stability of solitary and periodic waves of the $b$-family equation \eqref{eq b-family} has been studied quite extensively, the understanding of the local dynamics of these waves under the $b$-KP flow is much less developed. The only results that the authors are aware of regard the line solitary waves of the CH-KP equation: the nonlinear transverse instability of the solitary waves to the CH-KP-I equation is established in \cite{CJ21}, and linear stability of small-amplitude solitary waves is confirmed for CH-KP-II very recently \cite{GLP23}.

In this paper, we will investigate the transverse spectral stability/instability of small periodic traveling waves of the $b$-family equation \eqref{eq b-family} with respect to perturbations in the $b$-KP flow. Compared with the study of solitary waves, the stability of periodic waves is usually more delicate. The periodic waves in general exhibit a richer structural complexity, characterized by dependencies on three parameters -- namely, the period, wave speed, and integration constant. Such higher degree of freedom often introduces additional technical difficulties not encountered in the analysis of solitary waves. Moreover, a more broader class of perturbations can be considered for periodic waves, encompassing co-periodic, multiple-periodic, localized perturbations, among others. Our motivation for specifically studying the small-amplitude period waves is inspired by the work of Haragus \cite{Har11}, where perturbation arguments have been successfully employed to discern the spectra of the linear operator. In contrast to the stability analysis of large waves, where instability criteria can usually be derived (in the particular case of integrable systems, explicit computation can be performed, but \eqref{CHKP} is in general non-integrable), our choice to work with small-amplitude waves is motivated by the potential for obtaining more explicit information on the spectra, and for allowing for a broader range of perturbation types.

\subsection{Main results}

Although the basic idea of the approach stems from the work of Haragus \cite{Har11}, the quasilinear structure of the equation \eqref{CHKP} makes the spectral computation a lot more involved. For the case of $b$-KP-I flow $\sigma = -1$ with co-periodic perturbations in the direction of propagation (see Section \ref{sec formulation} for a precise definition of the class of perturbations and the corresponding notion of spectral stability), the linearized problem does not assume a natural Hamiltonian structure, and hence the standard index counting method cannot be applied directly. On the other hand, the linearized operator at the trivial solution does admit a decomposition into a composition of a skew-adjoint operator with a self-adjoint operator, and thus the counting result can be used to provide direct insights for the spectrum of this operator. Such information can then be transferred to the linearized operators at small-amplitude waves. Through a careful perturbation argument and explicit computation on the expansion of the spectra, we confirm the emergence of long-wave transverse instability for a range of $b$, which includes the examples of CH-KP ($b = 2$) and Degasperis--Procesi(DP)-KP ($b = 3$). More interestingly, depending on the wave number $k$ of the line periodic waves, there also exists a large region of values of $b$ where the line periodic waves are transversally spectrally stable.

In contrast to the $b$-KP-I equation, the spectral analysis for the $b$-KP-II equation ($\sigma = 1$) presents a more intricate challenge, reminiscent of the complexities found in the classical KP-II equation. The computation of the spectrum becomes substantially more complicated due to the nature of the dispersion relation, which is more likely to host unstable modes. Notably, in the limit of zero transverse wavelength, the dispersion relation may harbor an infinite number of potentially unstable eigenvalues. Tracking the locations of these eigenvalues further adds to the complexity, requiring the computation of the Taylor expansion of the corresponding eigen-matrix to an arbitrarily high order. The difficulties involved in these computations make it exceptionally challenging to achieve a comprehensive spectral analysis.
%On the other hand, like in the classical KP-II equation, the spectrum computation for the $b$-KP-II equation ($\sigma = 1$) seems much more complicated since the corresponding dispersion relation is more likely to host unstable modes. In fact, there could potentially be infinitely many unstable eigenvalues in the zero transverse wavelength limit. Moreover, to track the locations of these potentially unstable eigenvalues requires computing the Taylor expansion of the corresponding eigen-matrix up to an arbitrarily high order. All of these difficulties make it very challenging to obtain a complete spectral analysis.
What we are able to conclude in this case is a characterization of the spectra under long wavelength transverse perturbations. While a complete spectral analysis remains elusive, our findings align well with the spectral stability observed in the context of CH-KP-II solitary waves \cite{GLP23}.
%This result is in a good agreement with the spectral stability of the CH-KP-II solitary waves \cite{GLP23}.
%Our explicit calculation indicates that the spectrum distribution complements the one for the $b$-KP-I flow.
\begin{theorem}[Informal statement of transverse stability for periodic perturbation]
Let $b \ne -1$. Consider a $2\pi/k$-periodic traveling wave solution of \eqref{eq b-family} constructed in Lemma \ref{le1}.
\begin{enumerate}[label=\textup{(\alph*)}]
\item For $\sigma = -1$ ($b$-KP-I) and the amplitude of the wave is sufficiently small, such a wave is transversely spectrally unstable with respect to co-periodic perturbations in the $x$-direction and periodic in the $y$-direction, when the parameters $(b,k^2)$ lie outside the shaded region showed in Figure \ref{stab I figure}. This wave is transversely spectrally stable otherwise, provided that the wave is spectrally stable with respect to longitudinal
perturbations.
\item For $\sigma = 1$ ($b$-KP-II) and the amplitude of the wave is sufficiently small, such a wave is transversely spectrally unstable with respect to co-periodic perturbations in the $x$-direction and periodic in the $y$-direction, when the parameters $(b,k^2)$ lie inside the shaded region showed in Figure \ref{stab I figure}. Otherwise this wave is transversely spectrally stable under long-wave transverse perturbation, provided that the wave is spectrally stable with respect to longitudinal perturbations.
\end{enumerate}
\begin{figure}
  \centering
  \includegraphics[page=1,scale=0.8]{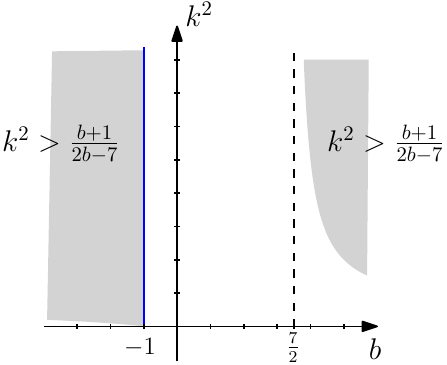}
  \vspace{-1.5ex}
  \caption{A schematic plot of the stability region for the periodic perturbation of the $b$-KP flow. The two shaded regions correspond to $\{(b, k^2) : b < -1, k^2 > \frac{b+1}{2b-7}\}$ and $\{(b, k^2) : b > \frac72, k^2 > \frac{b+1}{2b-7}\}$.
  \label{stab I figure} }
\end{figure}
\end{theorem}

For non-periodic perturbations in the direction of propagation, the linearized operator has bands of continuous spectra. Since the coefficients of the operator are periodic, we will use the classical Floquet-Bloch theory to replace the study of the invertibility of the original linearized operator by the invertibility of a family of Bloch operators in parameterized by the Floquet exponent (see Lemma \ref{lem5.1}). Through a detailed calculation of the spectrum of the linearized operator at the trivial solution (zero-amplitude solution), a perturbation argument is performed, which allows one to derive instability criterion for the $b$-KP-I case. We would like to point out that, a complete understanding of the Floquet analysis for the linearized operator is exceedingly difficult due to the appearance of the terms corresponding to the smoothing operator in the dispersion relation. Instead, when focused on the regime where the transverse perturbations are of finite wavelength, we manage to track the location where there is exactly one collision between a pair of eigenvalues of the linearized operator at the trivial solution from the imaginary axis, which results in the bifurcation of the unstable eigenvalues of the full linearized problem. The exact statement of the results is given in Theorem \ref{the5.1}. For long-wavelength transverse perturbations, an additional condition on the longitudinal wavelength is needed in order to eliminate the eigenvalue collisions. The detailed discussion is provided in Section \ref{subsec long}.

The remainder of this paper is organized as follows.  In Section \ref{sec existence}, we use the Lyapunov-Schmidt reduction to construct the family of the one-dimensional small-amplitude periodic traveling waves of the $b$-KP equation and provide a parameterization of these waves. In Section \ref{sec formulation}, we formulate the spectral problem for the $b$-KP equation and introduce the definition of spectral stability in various function space settings. In Sections \ref{sec per pert} and \ref{sec nonper}, we discuss the spectra of the resulting linear operators, and investigate the transverse spectral stability/instability of the small periodic waves of the $b$-KP-I and $b$-KP-II equation for periodic and non-periodic perturbations.
\subsection{Notations}
Throughout this paper, we will use the following notations. The space $L^2(\mathbb{R})$ denotes the set of real or complex-valued, Lebesgue measurable functions over $\mathbb{R}$ such that
\[
\|f\|_{L^2(\mathbb{R})}=\left(\int_{\mathbb{R}}|f(x)|^2 \mathrm{~d} x\right)^{1 / 2}<+\infty,
\]
and $L^2(\mathbb{T})$ denotes the space of $2 \pi$-periodic, measurable, real or complex-valued functions over $\mathbb{R}$ such that
\[
\| f \|_{L^2(\mathbb{T})}=\left(\frac{1}{2 \pi} \int_0^{2 \pi}|f(x)|^2 \mathrm{~d} x\right)^{1 / 2}<+\infty .
\]
The space $C_{\textup{bdd}}(\mathbb{R})$ contains all bounded continuous functions on $\mathbb{R}$, normed with
$$
\| f \|=\sup _{x \in \mathbb{R}}|f(x)| .
$$
For $s \in \mathbb{R}$, let $H^s(\mathbb{R})$ consist of tempered distributions such that
$$
\|f\|_{H^s(\mathbb{R})}=\left(\int_{\mathbb{R}}\left(1+|t|^2\right)^s|\widehat{f}(t)|^2 \mathrm{~d} t\right)^{1 / 2}<+\infty,
$$
where $\widehat f$ is the Fourier transform of $f$, and
$$
H^s(\mathbb{T})=\left\{f \in H^s(\mathbb{R}): f \text { is } 2 \pi \text {-periodic}\right\}.
$$
We define the $L^2(\mathbb{T})$-inner product as
\begin{equation}\label{1.2}
\langle f, g\rangle=\frac{1}{2 \pi} \int_0^{2 \pi} f(z) \bar{g}(z) \mathrm{d} z=\sum_{n \in \mathbb{Z}} \widehat{f}_n \overline{\widehat{g}}_n,
\end{equation}
where $\widehat{f}_n$ are Fourier coefficients of the function $f$ defined by
$$
\widehat{f}_n=\frac{1}{2 \pi} \int_0^{2 \pi} f(z) \mathrm{e}^{i n z} d z.
$$
We denote $\Re(\lambda)$ the real part of $\lambda \in \mathbb{C}$.

\section{Existence of small periodic traveling waves}\label{sec existence}
One-dimensional traveling waves of the $b$-KP equation \eqref{CHKP} are solutions of the form
\[
u(x, y, t) = u(x - ct),
\]
where $c>0$ is the speed of propagation, and $u$ satisfies the ODE
\[
\left[-c\left(u^{\prime}-u^{\prime \prime \prime}\right)+(b+1)u u^{\prime}+ \kappa u^{\prime}-b u^{\prime} u^{\prime \prime}-u u^{\prime \prime \prime}\right]^{\prime}=0.
\]
Integrating this equation twice, and writing $x$ instead of $x - ct$, we obtain the second order ODE
\[
(\kappa-c) u+c u^{\prime \prime}+\frac{b+1}{2} u^2-uu^{\prime\prime}-\frac{b-1}{2}\left(u^{\prime}\right)^2= A x-m(c- \kappa)^2,
\]
in which $A$ and $m$ are arbitrary integration constants. Considering periodic solutions, we set $A = 0$ and the equation reduces to
\begin{equation}\label{2.1}
(\kappa-c) u+c u^{\prime \prime}+\frac{b+1}{2} u^2-uu^{\prime\prime}-\frac{b-1}{2}\left(u^{\prime}\right)^2=-m(c- \kappa)^2.
\end{equation}
Since this equation does not possess scaling and Galilean invariance, we may not simply assume that $c=1, m=0$.

Let $u$ be a $2 \pi / k$-periodic function of its argument, for some $k>0$. Then, $w(z):=u(x)$ with $z=k x$, is a $2 \pi$-periodic function in $z$, satisfying
\begin{equation}\label{2.6}
(\kappa - c) w+ck^2 w_{z z}+\frac{b+1}{2} w^2-k^2 w w_{z z}-\frac{b-1}{2} k^2 w_z^2=-m(c - \kappa)^2.
\end{equation}
Let $F: H^2_{2\pi}(\mathbb{T}) \times\mr_{+}\times\mr\times\mr\rightarrow L^2(\mathbb{T})$ be defined as
\begin{equation}\label{2.7}
F(w ; k, c, m)=(\kappa - c) w+c k^2 w_{z z}+\frac{b+1}{2} w^2-k^2 w w_{z z}-\frac{b-1}{2} k^2 w_z^2+m(c - \kappa)^2.
\end{equation}
We seek a solution $w \in H^2(\mathbb{T})$ of
\begin{equation}\label{2.8}
F(w ; k, c, m) = 0.
\end{equation}
Noting that \eqref{2.7} remains invariant under $z\mapsto z+z_0$, $z\mapsto-z$ for any $z_0\in\mr$, we may assume that $w$ is even. Clearly $F$ is analytic on its arguments.
%Direct calculation shows that for $V \in H^2(\mathbb{T})$,
%\begin{align*}
%\partial_{w} F(w ; k, c, m)V = \left[ k^2 \partial z\left(c-(b-1)w\right) \partial z+(\kappa - c)+(b+1) w-k^2 w_{ z z}\right] V \in L^2(\mathbb{T}).
%\end{align*}
%and
%\begin{equation*}
%\begin{array}{ll}
%\p_k F(w ; k, c, m)=2 c k w_{z z}-2 k w w_{z z}-(b-1)k w_z^2, \\
%\p_c F(w ; k, c, m)=-w+k^2 w_{z z}+2 m(c - \kappa), \\
%\p_m F(w ; k, c, m) =\left( c - \kappa \right)^2.
%\end{array}
%\end{equation*}
%For all $k>0, c > 0$ and $\kappa, m \in\mr$,

It is easy to see that a constant solution $w_0$ of equation \eqref{2.8} satisfies
\begin{equation}\label{2.11}
%w_0(c, m, \kappa)=m(c - \kappa)+O\left(m^2\right).
\frac{b+1}{2} w_0^2 + (\kappa - c) w_0 + m(c - \kappa)^2 = 0.
\end{equation}
For $c \in \mr$ and $\kappa > 0$, when $b = -1$ we have
\begin{equation*}
w_0 = w_0(m) = \left\{\begin{aligned}
  & m (c - \kappa), \ & \text{ when } c \ne \kappa, \\
  & \text{any constant,} \ & \text{ when } c = \kappa.
\end{aligned}\right.
\end{equation*}
When $b \ne -1$, and $2m(b+1) < 1$, \eqref{2.11} is a genuine quadratic equation, and we find
\begin{equation*}
w_0(m) = \frac{c - \kappa \pm |c - \kappa| \sqrt{1 - 2m(b+1)}}{b+1}.
\end{equation*}
It follows from the implicit function theorem that if non-constant solutions of \eqref{2.8} (and hence \eqref{2.6}) bifurcate from $w = w_0$ for some $c = c_0$ then necessarily,
\begin{equation*}%\label{2.12}
L_0 := \p_w F\left(w_0 ; c_0, k, m\right): H^2(\mathbb{T}) \rightarrow L^2(\mathbb{T})
\end{equation*}
is not an isomorphism, where
\begin{equation*}%\label{2.13}
 L_0=k^2\left(c_0-(b-1)w_0\right) \partial_z^2+\left( \kappa -c_0\right)+(b+1) w_0.
\end{equation*}
Further calculation reveals that $L_0 e^{inz}=0, n\in\mathbb{Z}$, if and only if
\begin{align}\label{2.14}
%c_0&=\frac{1}{1+k^2 n^2}\left((b+1) w_0+k^2 n^2(b-1) w_0+ \kappa \right)=\frac{\kappa}{1+k^2 n^2}+w_0 \frac{(b+1)+ k^2 n^2(b-1)}{1+k^2 n^2}.
c_0 =\frac{\kappa}{1+k^2 n^2}+w_0 \frac{(b+1)+ k^2 n^2(b-1)}{1+k^2 n^2},
\end{align}
which, when plugging in the form of $w_0$, would lead to a solution
\[
c_0 = c_0(k,m),
\]
at least for $m$ sufficiently small.

Without loss of generality, we restrict our attention to $|n| = 1$, to further simplify the analysis, we take the constant $m = 0$, and consider the constant solution $w_0 = 0$. This way \eqref{2.14} leads to
\[
c_0 = c_0(k) = \frac{\kappa}{1+k^2}.
\]
In this case it is straightforward to verify that the kernel of $L_0$ : $H^2(\mathbb{T}) \rightarrow L^2(\mathbb{T})$ is two-dimensional and spanned by $e^{\pm iz}$. Moreover, the co-kernel of $L_0$ is two-dimensional. Therefore, $L_0$ is a Fredholm operator of index zero. One may then follow an idea similar to that of \cite{HP,HP1} to employ a Lyapunov-Schmidt reduction and construct a one parameter family of non-constant, even and smooth solutions of \eqref{2.1} near $w = w_0 = 0$ and $c=c_0(k)$. The small-amplitude expansion of these solutions is given as follows and the details are provided in Appendix \ref{app A}.

\begin{lemma}\label{le1}
For each $\kappa, k > 0$, and $m = 0$, there exists a family of small amplitude $2\pi/k$-periodic
traveling waves of \eqref{CHKP} %and, abusing notation
\begin{equation}\label{2.2}
w(k, a) := u\left(k\left(x-c(k, a) t\right)\right)
%w(k, a)=u\left(k\left(x-c(k, a) t\right)\right)=: u(k, a)(z)
\end{equation}
for $|a|$ sufficiently small; $w$ and $c$ depend analytically on $k$ and $a$,
$w$ is smooth, even and $2\pi$-periodic in $z$, and $c$ is even in $a$. Furthermore, as $a \to 0$,
\begin{align}
w(k, a)= &\  a \cos z+a^2 \left(A_0+A_2 \cos 2z\right)+a^3A_3\cos3z +O\left(a^4 \right), \label{2.3} \\
c(k, a) = &\ c_0 + a^2 c_2+O\left( a^4 \right), \label{2.4}
\end{align}
with
\begin{equation}\label{2.5}
\begin{array}{ll}
\displaystyle A_0=\frac{\left(1+k^2\right)}{4\kappa k^2}\left( {(b-3)}k^2-{(b+1)}\right),\quad A_2=\frac{{(b+1)}\left(1+k^2\right)^2}{12 \kappa k^2},\\\\
\displaystyle A_3=\frac{{(b+1)}\left(k^2+1\right)^3}{192\kappa^2 k^4}\left({{(2b+3)} k^2+(b+1)}\right),\\\\
\displaystyle c_0 = \frac{\kappa}{1+k^2}, \quad c_2=\frac{1}{\kappa}\left({\frac {-2b^2+11b-11}{24} k^2 + \frac {5b^2-11b-16}{24} -\frac{5(b+1)^2}{24k^2}}\right).
\end{array}
%\begin{array}{ll}
%\displaystyle A_0=\frac{\left(1+k^2\right)}{4K k^2}\left( {(b-3)}k^2- {(b+1)}\right),\quad A_2=\frac{ {(b+1)}\left(1+k^2\right)^2}{12K k^2},\\\\
%\displaystyle A_3=\frac{ {(b+1)}\left(k^2+1\right)^3}{192K^2k^4}\left({ {(2b+3)} k^2+(b+1)}\right),\\\\
%\displaystyle c_2=\frac{1}{K}\left( {k^2\frac {-2b^2+11b-11}{24}+\frac {5b^2-11b-16}{24}-\frac{5(b+1)^2}{24k^2}}\right),
%\end{array}
\end{equation}
\end{lemma}

\section{Formulation of the spectral problem}\label{sec formulation}
Linearizing the $b$-KP equation \eqref{CHKP} about its one-dimensional periodic traveling wave solution $w$ given in \eqref{2.3}, and considering the perturbations to $w$ of the form $w+\varepsilon v(t,z,y)$, we arrive that the equation
\begin{equation*}%\label{3.1}
\begin{aligned}
& k\left[\left(1-k^2\p_z^2\right)\left(v_t-k c v_z\right)+ \kappa kv_z+(b+1) k\left(w v\right)_z - k^3\left(w v_{z z}+(b-1)w_z v_z+w_{zz}v\right)_z\right]_z+\sigma v_{y y}=0.
\end{aligned}
\end{equation*}
Using change of variables and abusing notation $t\rightarrow kt$, $y\rightarrow ky$, we obtain
\begin{equation*}%\label{3.2}
\begin{aligned}
&\left[\left(1-k^2\p_z^2\right)\left(v_t- c v_z\right)+ \kappa v_z+(b+1)\left(w v\right)_z -
k^2\left(w v_{z z}+(b-1)w_z v_z+w_{zz}v\right)_z\right]_z+\sigma v_{y y}=0.
\end{aligned}
\end{equation*}
For $v(z, y, t)=\mathrm{e}^{\lambda t+i \ell y} V(z)$, we have
\begin{equation*}%\label{3.3}
\begin{aligned}
&
\left[\left(1-k^2\p_z^2\right)\left(\lambda V-cV_z\right)+ \kappa V_z+(b+1)(w V)_z - k^2\left(w V_{z z}+(b-1)w_z V_z+w_{z z}V\right)_z\right]_z-\sigma \ell^2 V=0.
\end{aligned}
\end{equation*}
The left-hand side of this equation defines the differential operator
\begin{align}\label{3.4}
\mathcal{T}_a(\lambda, \ell) V := & \left(1-k^2 \p_z^2\right) \p_z\left[\lambda V-c V_z+\left(1-k^2 \p_z^2\right)^{-1} \p_z\right.\nonumber\\
&\left.\left( \kappa +(b+1) w-k^2w_{z z}
-k^2(b-1)w_z\p_z -k^2w\p_z^2\right) V\right]-\sigma\ell^2V.
\end{align}
Clearly, the spectral stability problem concerns the invertibility of $\mathcal{T}_a(\lambda, \ell)$.

The longitudinal problem corresponds to perturbations with $\ell = 0$. In the particular cases of CH ($b=2$) equation and Degasperis--Procesi (DP) equation ($b=3$), the spectral and orbital stability for smooth periodic waves were obtained via inverse scattering \cite{lenells05} or by exploiting the variational characterization of the waves \cite{geyer22,geyer22arxiv}. This variational argument was further extended to treat the nonlinear orbital stability of periodic waves to the general $b$-CH family \cite{ehrman23arxiv} for all $b \ne 1$. The first approach relies substantially on the structure implication from the special values of $b$: equation \eqref{eq b-family} is completely integrable only for $b = 2, 3$. On the other hand, the variational approach utilizes the Hamiltonian structures. It turns out that the standard Hamiltonian formulation of the DP equation is amenable to the usual spectral stability theory \cite{geyer22arxiv}, whereas one needs to resort to the non-standard Hamiltonian formulation involving momentum density for the CH \cite{geyer22} and for the general $b$-family \cite{ehrman23arxiv} to deduce the stability criterion for periodic waves.

We consider in this paper two dimensional transverse perturbations which require $\ell \ne 0$. Specifically, three types of perturbations will be addressed:
\begin{itemize}
\item periodic (in $z$) perturbations, where $\mathcal{T}_a(\lambda, \ell)$ is considered to be $H^4(\mathbb{T}) \to L^2(\mathbb{T})$,
\item localized perturbations, where $\mathcal{T}_a(\lambda, \ell)$ is considered to be $H^4(\mathbb{R}) \to L^2(\mathbb{R})$, and
\item bounded perturbations, where $\mathcal{T}_a(\lambda, \ell)$ is considered to be $C^4_{\textup{bdd}}(\mathbb{R}) \to C_{\textup{bdd}}(\mathbb{R})$.
\end{itemize}

The precise definition of the transverse spectral stability is given as follows.
\begin{definition}[Transverse spectral stability]
For a $2 \pi / k$-periodic traveling wave solution $u(x, y, t)=w(k(x-c t))$ of \eqref{CHKP}
%Assume that the $2 \pi / k$-periodic traveling wave solution $u(x, y, t)=w(k(x-c t))$ of \eqref{CHKP} is a stable solution of the one-dimensional equation \eqref{eq b-family},
where $w$ and $c$ are as in \eqref{2.3} and \eqref{2.4}, we say that the periodic wave $w$ is transversely spectrally stable with respect to two-dimensional periodic perturbations (resp. non-periodic (localized or bounded perturbations)) if the $b$-KP operator $\mathcal{T}_a(\lambda, \ell)$ acting in $L^2(\mathbb{T})$ (resp. $L^2(\mathbb{R})$ or $C_{\textup{bdd}}(\mathbb{R})$ ) with domain $H^{4}(\mathbb{T})$ (resp. $H^{4}(\mathbb{R})$ or $C^4_{\textup{bdd}}(\mathbb{R})$) is invertible, for any $\lambda \in \mathbb{C}, \Re(\lambda)>0$ and any $\ell \neq 0$.
\end{definition}

%{\bf check 1D stability results by Pelinovsky et al.}
%
%Depending on the space in which we are studying the invertibility of $\mathcal{T}_a(\lambda, \ell)$, we split our study into periodic $\left(L^2(\mathbb{T})\right)$ and non-periodic perturbations $\left(L^2(\mathbb{R})\right.$ or $\left.C_{\textup{bdd}}(\mathbb{R})\right)$. Also, depending upon the values of $\ell$ we distinguish two different regimes: long-wavelength transverse perturbations, when $|\ell| \ll 1$ and short or finite wavelength transverse perturbations, otherwise.

\section{Periodic perturbations}\label{sec per pert}
In this section we study the transverse spectral stability of the periodic waves $w$ with respect to periodic perturbations for the $b$-KP equation. More precisely, we study the invertibility of the operator $\mathcal{T}_{a}(\lambda,\ell)$ acting in $L^2(\mathbb{T})$ with domain $H^4(\mathbb{T})$ for $\lambda\in \mathbb{C}, \Re(\lambda) >0$ and $\ell\in\mr \backslash \{ 0 \}$. Following the general strategy of \cite{Har11}, let's first reformulate the spectral problem for this particular case, as in the proposition below.
\begin{proposition}\label{pro4.1}
The following statements are equivalent:
\begin{enumerate}
\item  $\mathcal{T}_{a}(\lambda,\ell)$ acting in $L^2(\mathbb{T})$ with domain $H^4(\mathbb{T})$ is not invertible.
\item  The restriction of $\mathcal{T}_{a}(\lambda,\ell)$ to the subspace $L_0^2(\mathbb{T})$ of $L^2(\mathbb{T})$ is not invertible, where
    \[
    L_0^2(\mathbb{T})=\left\{f \in L^2(\mathbb{T}): \int_0^{2 \pi} f(z) d z=0\right\}.
    \]
\item  $\lambda$ belongs to the spectrum of the operator $\mathcal{A}_{a}(\ell)$ acting in $L_0^2(\mathbb{T})$ with domain $H^1(\mathbb{T}) \cap L_0^2(\mathbb{T})$ where $\mathcal{A}_{a}(\ell)$ is defined as follows:
\begin{equation}\label{3.11}
\begin{aligned}
\mathcal{A}_{a}(\ell) := & \ \p_z \left[c-\left(1-k^2 \p_z^2\right)^{-1}\right.\nonumber\\
& \quad \ \left.\left(\kappa +(b+1) w-k^2w_{z z}
-k^2(b-1)w_z\p_z -k^2w\p_z^2-\sigma\ell^2 \p_z^{-2}\right)\right]\nonumber\\
= & \ \p_z\left(1-k^2 \p_z^2\right)^{-1} \left[c\left(1-k^2 \p_z^2\right)\right.\nonumber\\
& \quad \ \left.-\left(\kappa +{ {(b+1)}} w-k^2w_{z z}
-k^2{{(b-1)}}w_z\p_z -k^2w\p_z^2-\sigma\ell^2 \p_z^{-2}\right)\right].
\end{aligned}
\end{equation}
\end{enumerate}
\end{proposition}
The proof of the above result follows along similar lines as \cite[Lemma 4.1, Corollary 4.2]{Har11}, together with the fact that $1 - k^2 \partial_z^2 : H^{s+2}(\mathbb{T}) \to H^s(\mathbb{T})$ is invertible. Therefore it suffices to analyze the spectrum of $\mathcal{A}_{a}(\ell)$. Since it has a compact resolvent, the spectrum consists of isolated eigenvalues with finite algebraic multiplicity. Moreover, the symmetry of $w$ leads to the following symmetry of the spectrum of $\mathcal{A}_{a}(\ell)$, the proof of which follows along the same line as \cite[Lemma 4.3]{Har11}, and hence we omit it.
\begin{lemma}\label{lem4.1}
The spectrum of $\mathcal{A}_{a}(\ell)$ is symmetric with respect to both the real and imaginary axes.
\end{lemma}

The operator $\mathcal{A}_{0}(\ell)$ has constant coefficients, and a straightforward calculation reveals that
\begin{equation*}%\label{3.13}
\mathcal{A}_{0}(\ell) \mathrm{e}^{i n z}=i \omega_{n, \ell} \mathrm{e}^{i n z} \quad \text { for all }\quad n \in \mathbb{Z}^*:=\mathbb{Z} \backslash\{0\},
\end{equation*}
where
\begin{align*}%\label{3.12}
 \omega_{n, \ell} & = n\left(c_0-\frac{\kappa}{1+k^2 n^2}-\frac{\sigma\ell^2}{n^2\left(1+k^2 n^2\right)}\right)\nonumber\\
 & = n\left(\frac{\kappa}{1+k^2}-\frac{\kappa}{1+k^2 n^2}-\frac{\sigma\ell^2}{n^2\left(1+k^2 n^2\right)}\right)
 \nonumber\\
& = \frac{n}{1+k^2 n^2}\left(\kappa \frac {k^2(n^2-1)}{1+k^2}-\frac{\sigma\ell^2}{n^2}\right).
\end{align*}
Consequently, the $L_0^2(\mathbb{T})$-spectrum of $\mathcal{A}_{0}(\ell)$ consists of purely imaginary eigenvalues of finite multiplicity. On the other hand, we write
\[
\mathcal{A}_{a}(\ell)=\mathcal{A}_{0}(\ell)+\widetilde{\mathcal{A}}_a,
\]
with
\begin{equation}\label{3.15}
\mathcal{A}_{0}(\ell)=\p_z\left(1-k^2 \p_z^2\right)^{-1}\left(-c_0k^2\p_z^2+ c_0- \kappa + \sigma \ell^2 \p_z^{-2}\right)
\end{equation}
and
\begin{align*}%\label{3.09}
\widetilde{\mathcal{A}}_a
&=\mathcal{A}_{a}(\ell)-\mathcal{A}_{0}(\ell)\nonumber\\
&=\p_z\left(1-k^2 \p_z^2\right)^{-1}\left[(c-c_0)\left(1-k^2 \p_z^2\right)\right.\nonumber\\
&\quad\left.-\left((b+1) w-k^2w_{z z}
-k^2(b-1)w_z\p_z -k^2w\p_z^2\right)\right].
\end{align*}
A direct calculation shows that
\begin{equation}\label{3.08}
\|\widetilde{\mathcal{A}}_a\|_{H^1(\mathbb{T})\rightarrow L^2(\mathbb{T})}=O(|a|) \quad \text{as } a \to 0.
\end{equation}
%as $a\rightarrow 0$ uniformly in the operator norm.
A standard perturbation argument ensures that the spectra of $\mathcal{A}_{a}(\ell)$ and $\mathcal{A}_{0}(\ell)$ stay close for $|a|$ small. Due to the symmetry in Lemma \ref{lem4.1}, it follows that for $|a|$ sufficiently small the bifurcation of eigenvalues of $\mathcal{A}_{a}(\ell)$ from the imaginary axis happens in pairs, and is completely due to the collisions of eigenvalues of $\mathcal{A}_{0}(\ell)$ on the imaginary axis.

Note that the operator $\mathcal{A}_{a}(\ell)$ admits a natural decomposition
\[
\mathcal{A}_{a}(\ell) = \mathcal{J} \mathcal{K}_{a}(\ell),
\]
where $\mathcal{J} := \p_z\left(1-k^2 \p_z^2\right)^{-1}$ is skew-adjoint and invertible in $L^2_0(\mathbb{T})$, and
\[
\mathcal{K}_{a}(\ell) := c\left(1-k^2 \p_z^2\right) - \big[ \kappa +{ {(b+1)}} w-k^2w_{z z}
-k^2{{(b-1)}}w_z\p_z -k^2w\p_z^2-\sigma\ell^2 \p_z^{-2} \big].
\]
However, it can be checked that, except for $b = 2, 3$, the operator $\mathcal{K}_{a}(\ell)$ fails to be self-adjoint. Therefore the standard index counting for Hamiltonian system does not immediately apply to $\mathcal{A}_{a}(\ell)$.

The way to go around this issue is to investigate the spectrum of the operator $\mathcal{A}_{0}(\ell)$, and then use perturbation method to transfer the spectral information to $\mathcal{A}_{a}(\ell)$. It turns out that we can decompose the operator $\mathcal{A}_{0}(\ell)$ into a composition of $\mathcal{J}$ and a self-adjoint operator $\mathcal{K}_{0}(\ell)$:
\[
\mathcal{A}_{0}(\ell) = \mathcal{J} \mathcal{K}_{0}(\ell),
\]
where
\[
\mathcal{K}_{0}(\ell)=c_0-c_0k^2\p_z^2- \kappa + \sigma \ell^2 \p_z^{-2}
\]
and $c_0$ is given in \eqref{2.5}.

Standard linear Hamiltonian theory suggests to track the Krein signature to detect the onset of  instability bifurcation. Specifically, the Krein signature $K_n$ of an eigenvalue $i\omega_{n, \ell}$ of $\mathcal{A}_{0}(\ell)$ is defined as
\begin{equation}\label{3.16}
\begin{aligned}
K_n :=\operatorname{sgn}\left(\left\langle\mathcal{K}_0(\ell) \mathrm{e}^{i n z}, \mathrm{e}^{i n z}\right\rangle\right)
=\operatorname{sgn}\left( \kappa \frac {k^2(n^2-1)}{1+k^2}-\frac{\sigma\ell^2}{n^2}\right).
\end{aligned}
\end{equation}
A necessary condition for a pair of eigenvalues to leave imaginary axis after collision is that they carry opposite Krein signatures.

%Now we discuss the $b$-KP-I equation and $b$-KP-II equation respectively.
\subsection{$b$-KP-I equation}\label{se4.1}
We first consider the case $\sigma= -1$. For this we compute to get
\[
\operatorname{spec}_{L_0^2(\mathbb{T})}\left(\mathcal{K}_{0}(\ell)\right)
=\left\{ \kappa \frac {k^2(n^2-1)}{1+k^2}+\frac{\ell^2}{n^2}\;;\;n\in \mathbb{Z}^*\right\}.
\]

\subsubsection{Finite and short wavelength transverse perturbations}
It's easy to see from the above that when $\sigma = -1$, the Krein signatures of all eigenvalues remain the same, which implies that for $|a|$ sufficiently small, the eigenvalues will not bifurcate from the imaginary axis even if there is a collision away from the origin. On the other hand, the only possible scenario when eigenvalues split into the complex plane as unstable eigenvalues is when $\ell$ is small and the collision occurs at the origin. Therefore we have the following lemma.
% there exists an $\ell^*>0$ depending on $a$ such that for all $|\ell| > \ell^*$, the spectrum of $\mathcal{A}_{a}(\ell)$ is purely imaginary. In summary, 
\begin{lemma}\label{lem4.3}
For any given $\ell^*>0$ there exists $|a|$ sufficiently small such that for all $|\ell| > \ell^*$, the spectrum of $\mathcal{A}_{a}(\ell)$ is purely imaginary.
\end{lemma}

\subsubsection{Long wavelength transverse perturbations}\label{4.1.2}
The discussion above leaves possible the onset of instability due to eigenvalue coalescence at the origin, for small $\ell$. This corresponds to the transverse perturbations being of long wavelength.

Different from how we obtain Lemma \ref{lem4.3}, now we will perform a double perturbation by regarding $\mathcal{A}_{a}(\ell)$ as a perturbation of the constant-coefficient operator
\begin{equation*}%\label{3.17}
\mathcal{A}_{0}(0)=\p z\left(1-k^2 \p_z^2\right)^{-1} \mathcal{K}_{0}(0)
=\p z\left(1-k^2 \p_z^2\right)^{-1}\left(-c_0k^2\p_z^2+ c_0- \kappa \right)
\end{equation*}
acting in $L_0^2(\mathbb{T})$. A direct calculation shows that the spectrum of $\mathcal{A}_{0}(0)$ is given by
\begin{equation}\label{spec A_0(0)}
\operatorname{spec}_{L_0^2(\mathbb{T})}\left(\mathcal{A}_{0}(0)\right)
=\left\{in  r_*(n) \;;\;n\in \mathbb{Z}^*\right\}, \quad \text{where} \quad r_*(n) := \kappa \left(\frac{1}{1+k^2}-\frac{1}{1+k^2 n^2}\right).
\end{equation}
In particular, zero is a double eigenvalue of $\mathcal{A}_{0}(0)$, and the remaining eigenvalues are all simple, purely imaginary, and located outside the open ball $B\left(0; r_*(2) \right)$. Besides, letting
\begin{equation*}%\label{3.18}
\widehat{\mathcal{A}}_a(\ell) := \mathcal{A}_a(\ell)-\mathcal{A}_0(0)=\widetilde{\mathcal{A}}_a + \p z\left(1-k^2 \p_z^2\right)^{-1}(\sigma\ell^2\p_z^{-2}),
\end{equation*}
from \eqref{3.08}, we get
\[
\left\|\widehat{\mathcal{A}}_a(\ell)\right\|_{H^1\rightarrow L^2} = O\left(\ell^2+|a|\right).
\]

We proceed similarly to the proof in \cite[Lemma 4.7]{Har11} to get the following lemma.
\begin{lemma}\label{lem4.2}
The following properties hold, for any $\ell$ and $a$ sufficiently small.
\begin{enumerate}[label=\textup{(\alph*)}]
\item The spectrum of $\mathcal{A}_a(\ell)$ decomposes as
\[
\operatorname{spec}_{L_0^2(\mathbb{T})}\left(\mathcal{A}_a(\ell)\right)
=\operatorname{spec}_0\left(\mathcal{A}_a(\ell)\right) \cup \operatorname{spec}_1\left(\mathcal{A}_a(\ell)\right),
\]
with
\begin{equation*}
\begin{aligned}
&\operatorname{spec}_0\left(\mathcal{A}_a(\ell)\right) \subset B\left(0 ; \frac{r_*(2)}{2}\right), \qquad \operatorname{spec}_1\left(\mathcal{A}_a(\ell)\right)
\subset \mathbb{C} \backslash \overline{B\left(0; r_*(2) \right)},
\end{aligned}
\end{equation*}
where $r_*(n)$ is defined in \eqref{spec A_0(0)}. 
%\[
%r^* := r_*(2) = \kappa \left(\frac{1}{1+k^2}-\frac{1}{1+4 k^2 }\right).
%\]
\item The spectral projection $\Pi_a(\ell)$ associated with $\operatorname{spec}_0\left(\mathcal{A}_a(\ell)\right)$ satisfies $\left\|\Pi_a(\ell)-\Pi_0(0)\right\|\\=O\left(\ell^2+|a|\right)$.
\item The spectral subspace $\mathcal{X}_a(\ell)=\Pi_a(\ell)\left(L_0^2(\mathbb{T})\right)$ is two dimensional.
\end{enumerate}
\end{lemma}
This lemma ensures that for sufficiently small $\ell$ and $a$, bifurcating eigenvalues from the origin are uniformly separated from the rest of the spectrum.  In the following lemma, we show that for sufficiently small $\ell$ and $a$, the two eigenvalues in $\operatorname{spec}_0\left(\mathcal{A}_a(\ell)\right)$ leave imaginary axes.
\begin{theorem}\label{the4.1}
Assume $|\ell|, |a|$ are sufficiently small. Denote
\begin{equation}\label{def l_a}
\ell^2_a := \left[(b+1)+(7-2b)k^2\right]\frac{(b+1)\left(1+k^2\right)^2}{12 \kappa k^2}a^2.
\end{equation}
%\[
%\ell^2_a :=
%\left[ (b+1)+(7-2b)k^2\right] \frac{(b+1)\left(1+k^2\right)^2}{12 \kappa k^2}a^2+O\left(a^4\right),
%\]
If $\ell^2_a>0$, then there exists some
\[
\ell_*^2 := \ell^2_a + O(a^4) > 0
\] such that
\begin{enumerate}[label=\textup{(\roman*)}]
\item for any $\ell^2 > \ell_*^2$, the spectrum of $\mathcal{A}_a(\ell)$ is purely imaginary.
\item for any $\ell^2< \ell_*^2$, the spectrum of $\mathcal{A}_a(\ell)$ is purely imaginary, except for a pair of simple real eigenvalues with opposite signs.
\end{enumerate}
If $\ell_a^2 < 0$, then the spectrum of $\mathcal{A}_a(\ell)$ is purely imaginary. % for sufficiently small $|a|$.
\end{theorem}

\begin{proof}Consider the decomposition of the spectrum of $\mathcal{A}_a(\ell)$ in lemma \ref{lem4.2}. We first study $\operatorname{spec}_1\left(\mathcal{A}_a(\ell)\right)$ for $\ell$ and $a$ sufficiently small. For $n \in \mathbb{Z}^* \backslash\{\pm 1\}$, as $\ell$ is sufficiently small, $K_n = 1$ in \eqref{3.16} for all $n$. This implies that even if eigenvalues in $\operatorname{spec}_1\left(\mathcal{A}_0(\ell)\right)$ collide, they remain on the imaginary axis. Then for all $\ell$ and $a$ sufficiently small,  ${\operatorname{spec}}_1\left(\mathcal{A}_a(\ell)\right)$ is a subset of the imaginary axis.

The eigenvalues in $\operatorname{spec}_0\left(\mathcal{A}_a(\ell)\right)$ are the eigenvalues of the restriction of $\mathcal{A}_a(\ell)$ to the two-dimensional spectral subspace $\mathcal{X}_a(\ell)$. We determine the location of these eigenvalues by computing successively a basis of $\mathcal{X}_a(\ell)$, the $2 \times 2$ matrix representing the action of $\mathcal{A}_a(\ell)$ on this basis, and the eigenvalues of this matrix.

For $a=0$, $\mathcal{A}_0(\ell)$ is an operator with constant coefficients, and
\[
\operatorname{spec}_0\left(\mathcal{A}_0(\ell)\right)
=\left\{i\frac{ \ell^2}{ 1 + k ^ { 2 } },-i\frac{ \ell^2}{1 + k ^ { 2 } }\;;\;n\in \mathbb{Z}^*\right\}.
\]
The associated eigenvectors are $e^{i z}$ and $e^{-i z}$, and we choose
\[
\xi_0^0(\ell)=\sin (z), \quad \xi_0^1(\ell)=\cos (z),
\]
as basis of the corresponding spectral subspace. Since
$$
\mathcal{A}_0(\ell) \xi_0^0(\ell)=\frac{ \ell^2}{ 1 + k ^ { 2 } } \xi_0^1(\ell), \quad \mathcal{A}_0(\ell) \xi_0^1(\ell)=-\frac{ \ell^2}{1 + k ^ { 2 }} \xi_0^0(\ell),
$$
the $2 \times 2$ matrix representing the action of $\mathcal{A}_0(\ell)$ on this basis is given by
$$
M_0(\ell)=\left(\begin{array}{cc}
0 & \frac{ \ell^2}{ 1 + k ^ { 2 } } \\
-\frac{ \ell^2}{ 1 + k ^ { 2 } } & 0
\end{array}\right).
$$

%On the other hand, a direct calculation shows that zero is an $L_0^2(\mathbb{T})$-eigenvalue of $\mathcal{A}_a(0)$ of multiplicity two with eigenfunctions $\left(\partial_m c\right)\left(\partial_a w\right)-$ $\left.\left(\partial_a c\right)\left(\partial_m w\right)\right)(z ; k, a, 0)$ and $\partial_z w(z ; k, a, 0)$. Indeed, differentiating \eqref{2.6} with respect to $z, a, m$ and evaluating at $m = 0$, we find that
%\begin{equation*}%\label{4.9}
%\mathcal{A}_a(0)\left(\partial_z w\right)=0, \;\; \mathcal{A}_a(0)\left(\partial_a w\right)=\left(\partial_a c\right)\left(\partial_z w\right), \;\; \mathcal{A}_a(0)\left(\partial_m w\right)=\left(\partial_m c\right)\left(\partial_z w\right),
%\end{equation*}
%respectively (see \cite[Lemma 3.1]{Hu} for details). 
We use expansions of $w$ and $c$ in \eqref{2.3} and \eqref{2.4} to calculate the expansion of a basis $\left\{\xi_a^0(z), \xi_a^1(z)\right\}$ for $\mathcal{X}_a(\ell)$ for small $a$ and $\ell$ as
\begin{align*}
\xi_a^0(z) := & -\frac{1}{a} \partial_z w(z ; k, a, 0)=\sin z+2 a A_2 \sin 2 z+3 a^2 A_3 \sin 3 z+O\left(a^3\right), \\ % \label{4.10}\\
\xi_a^1(z) := & -\frac{(1+k^2)^2}{\kappa k^2\left((b+1)+(b-1)k^2\right)}\left(\left(\partial_m c\right)\left(\partial_a w\right)-\left(\partial_a c\right)\left(\partial_m w\right)\right)(z ; k, a, 0) \nonumber\\
= & \cos z+2 a A_2 \cos 2 z+3 a^2 A_3 \cos 3 z+O\left(a^3\right). %\label{4.11}
\end{align*}
As
\[
\begin{aligned}
\mathcal{A}_a(0)=&\left(1-k^2 \p_z^2\right)^{-1} \p_z\\
&\left[c\left(1-k^2 \p_z^2\right)-\left( \kappa +(b+1) w-k^2w_{z z}
-k^2(b-1)w_z\p_z -k^2w\p_z^2\right)\right]\\
=&\left(1-k^2 \p_z^2\right)^{-1} \p_z\\
&\left[-k^2\p_z(c-w)\p_z+\left(c- \kappa -(b+1) w+k^2w_{z z}
+(b-2)k^2w_z\p_z\right)\right],
\end{aligned}
\]
using expansions of $w$ and $c$ in \eqref{2.3} and \eqref{2.4}, we obtain that the action of $\mathcal{A}_a(0)$ on the basis $(\xi_a^0(z), \xi_a^1(z))$ is
\begin{equation*}%\label{4.12}
M_a(0)=\left(\begin{array}{cc}
0 & O\left( |a|^3 \right) \\\\
\displaystyle A_2 \frac{(b+1)+(7-2b)k^2}{1+k^2}a^2 + O\left( |a|^3 \right) & 0
\end{array}\right).
\end{equation*}
Together with the expression of $M_0(\ell)$, we get that
\begin{equation*}%\label{4.13}
M_a(\ell)=\left(\begin{array}{cc}
0 & \displaystyle  \frac{ \ell^2}{ 1 + k ^ { 2 } } + O\left(|a|\left(\ell^2+a^2\right)\right) \\\\
\displaystyle -\frac{ \ell^2}{ 1 + k ^ { 2 } }+A_2\frac{(b+1)+(7-2b)k^2}{1+k^2}a^2 + O\left(|a|\left(\ell^2+a^2\right)\right) & 0
\end{array}\right).
\end{equation*}
The two eigenvalues of $M_a(\ell)$, which are also the eigenvalues in $\operatorname{spec}_0\left(\mathcal{A}_a(\ell)\right)$, are roots of the characteristic polynomial
$$
P(\lambda)=\lambda^2+\frac{ \ell^2}{ 1 + k ^ { 2 } }\left(\frac{ \ell^2}{ 1 + k ^ { 2 } }-\frac{(b+1)+(7-2b)k^2}{1+k^2}A_2a^2 \right)+O\left(|a| \ell^2\left(\ell^2+a^2\right)\right)
$$
and we conclude that
\begin{equation}\label{4.14}
\begin{split}
\lambda^2=&-\frac{ \ell^2}{ 1 + k ^ { 2 } }\left(\frac{ \ell^2}{ 1 + k ^ { 2 } }-\frac{(b+1)+(7-2b)k^2}{1+k^2}A_2a^2\right)+O\left(a^2 \ell^2\left(\ell^2+a^2\right)\right) \\
=&-\frac{ \ell^2}{ 1 + k ^ { 2 } }\left(\frac{ \ell^2}{ 1 + k ^ { 2 } }-\frac{(b+1)+(7-2b)k^2}{1+k^2}\frac{(b+1)\left(1+k^2\right)^2}{12 \kappa k^2}a^2\right) \\
&+O\left(a^2 \ell^2\left(\ell^2+a^2\right)\right) \\
=&-\frac{ \ell^2}{ (1 + k ^ { 2 })^2 }\left[ \ell^2 - \ell^2_a + O\left(a^2\left(\ell^2+a^2\right)\right) \right],
\end{split}
\end{equation}
where $\ell_a^2$ is defined in \eqref{def l_a}. Thus when $\ell^2_a > 0$, then for $a$ sufficiently small there exists some
\[
\ell_*^2 := \ell^2_a + O(a^4) > 0
\]
such that when $\ell^2 > \ell_*^2$ then the two eigenvalues are purely imaginary, whereas the eigenvalues are real with opposite signs when $\ell^2<\ell_*^2$. On the other hand, when $\ell^2_a < 0$ then the spectrum of $\mathcal{A}_a(\ell)$ is purely imaginary for $a$ sufficiently small.
%Consequently, for any $a$ sufficiently small there exists a value $\ell_a \in \mathbb{C}$ where
%\[
%\ell_a^2 := \left((b+1)+(7-2b)k^2\right)\frac{(b+1)\left(1+k^2\right)^2}{12 \kappa k^2}a^2+O\left(a^4\right) \in \mathbb{R},
%\]
%such that for $\ell_a^2>0$ the two eigenvalues are purely imaginary when $\ell^2 > \ell_a^2>0$, and real with opposite signs when $\ell^2<\ell_a^2$. For $\ell_a\leq0$, then the spectrum of $\mathcal{A}_a(\ell)$ is purely imaginary.
\end{proof}
We list the cases of $\ell_a^2>0$ and $\ell_a^2 < 0$ in the following lemma and the detail discussion is given in the Appendix \ref{app B}.
\begin{lemma}\label{lem4.4}
$\ell_a^2 >0$ is valid for the following cases:
\[
\textup{(1)} \ -1< b\leq\frac 7 2; \qquad \textup{(2)} \ b>\frac 7 2, \ k^2<\frac {b+1} {2b-7}; \qquad \textup{(3)} \ b<-1, \ k^2< \frac{b+1} {2b-7}.
\]
On the other hand, %$\ell_a^2 \leq0$, that is,
%\begin{enumerate}
%  \item $-1< b\leq\frac 7 2$;
%  \item $b>\frac 7 2$, $k^2<\frac {b+1} {2b-7}$;
%  \item $b<-1$, $k^2< \frac{b+1} {2b-7}$;
%\end{enumerate}
$\ell_a^2 < 0$ when
\[
\textup{(1)} \ b>\frac 7 2, \ k^2 > \frac {b+1} {2b-7}; \qquad \textup{(2)} \ b<-1, \ k^2 > \frac {b+1} {2b-7}. %; \qquad \textup{or \ \ (3)} \ b = -1.
\]
%\begin{enumerate}
%  \item $b>\frac 7 2$, $k^2\geq \frac {b+1} {2b-7}$;
%  \item $b<-1$, $k^2\geq \frac {b+1} {2b-7}$;
%  \item $b=-1$.
%\end{enumerate}
\end{lemma}

\subsection{$b$-KP-II equation}
Now let's turn to the case when $\sigma = 1$.  We begin by analyzing the spectra of the unperturbed operators $\mathcal{K}_{0}(\ell)$ and $\mathcal{A}_{0}(\ell)$.
\subsubsection{Spectrum of $\mathcal{K}_{0}(\ell)$}
Using Fourier series we find that the spectrum of the operator $\mathcal{K}_0(\ell)$ acting in $L_0^2(\mathbb{T})$ is given by (see also Figure \ref{fg1} (a))
\[
\operatorname{spec}_{L_0^2(\mathbb{T})}\left(\mathcal{K}_{0}(\ell)\right)
=\left\{ \kappa \frac {k^2(n^2-1)}{1+k^2}-\frac{\ell^2}{n^2}\;;\;n\in \mathbb{Z}^*\right\}.
\]

\begin{figure}[htbp]%\label{fg1}
\centering
\subfigure[]
{\begin{minipage}[t]{0.48\textwidth}
\centering
\includegraphics[width=3in]{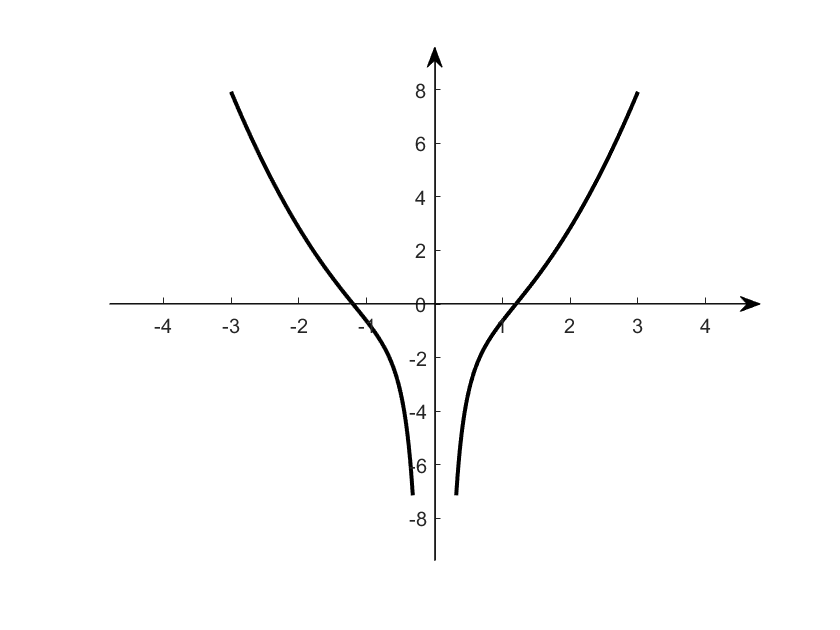}
\end{minipage}
}
\subfigure[]
{\begin{minipage}[t]{0.48\textwidth}
\centering
\includegraphics[width=3in]{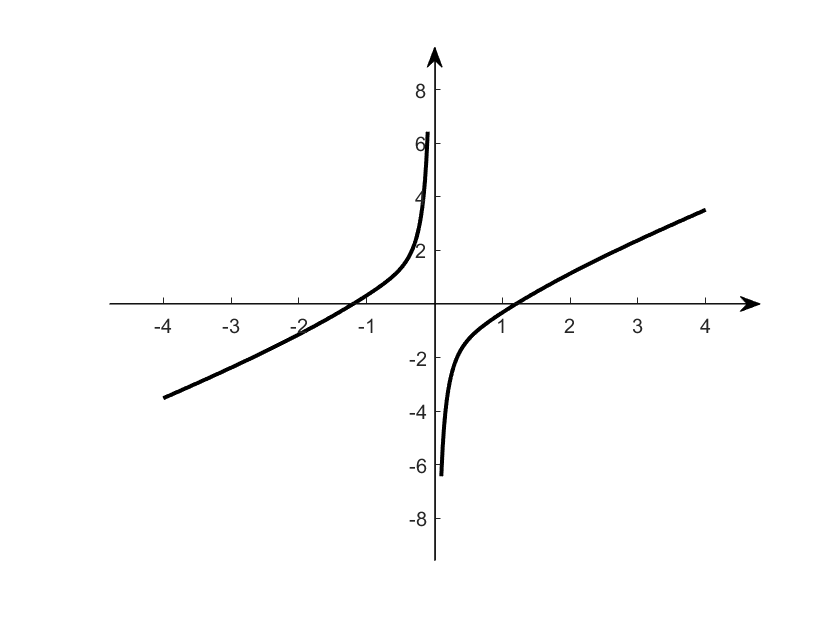}
\end{minipage}
}
\caption{[$b$-KP-II equation] (a) Graph of the map $n \mapsto \kappa \frac{k^2(n^2-1)}{1+k^2}- \frac{ \ell^2}{n ^ { 2 }}$ for $\kappa =2, k=1$ and $\ell=0.8$. The eigenvalues of $\mathcal{K}_0(\ell)$ are found by taking $n=q, q \in \mathbb{Z}^*$. (b) Graph of the dispersion relation $n \mapsto \nu(n) = n\left(\frac{\kappa}{1+k^2}-\frac{\kappa}{1+k^2 n^2}- \frac{ \ell^2}{n ^ { 2 } ( 1 + k ^ { 2 } n ^ { 2 })}\right)$ for $\kappa =2, k=1$ and $\ell=0.8$. The imaginary parts of the eigenvalues of $\mathcal{A}_0(\ell)$ are found by taking $k=n, n \in \mathbb{Z}^*$. Notice that the zeros of the two maps are the same. \label{fg1}}
\end{figure}
In contrast to the $b$-KP-I equation, the spectrum of $\mathcal{K}_0(\ell)$ contains negative eigenvalues, and the number of these eigenvalues increases with $\ell$.

\subsubsection{Spectrum of $\mathcal{A}_{0}(\ell)$}
The spectrum of the operator $\mathcal{A}_0(\ell)$ acting in $L_0^2(\mathbb{T})$ is given by (see also Figure \ref{fg1} (b))
\[
\operatorname{spec}_{L_0^2(\mathbb{T})}\left(\mathcal{A}_0(\ell)\right)
=\left\{i\nu_n(\ell)=in\left(\frac{\kappa}{1+k^2}-\frac{\kappa}{1+k^2 n^2}
-\frac{ \ell^2}{n ^ { 2 } ( 1 + k ^ { 2 } n ^ { 2 })}\right)\;:\;n\in \mathbb{Z}^*\right\}.
\]
Notice that the dispersion relation
\[
\nu(n) := n\left(\frac{\kappa}{1+k^2}-\frac{\kappa}{1+k^2 n^2}
-\frac{ \ell^2}{n ^ { 2 } ( 1 + k ^ { 2 } n ^ { 2 })}\right)
\]
is monotonically increasing on $(-\infty,-1]$ and $[1, \infty)$, so that colliding eigenvalues correspond to Fourier modes with opposite signs. A direct calculation then shows that for any $p, q \in \mathbb{N}^*$ the eigenvalues corresponding to the Fourier modes $p$ and $-q$ collide when
$$
\ell^2=\ell_{p, q}^2 := \kappa \frac{p q\left(1+k^2 p^2\right)\left(1+k^2 q^2\right)}{p\left(1+k^2 p^2\right)+q\left(1+k^2 q^2\right)}\left(\frac{p+q}{1+k^2}-\frac{p\left(1+k^2 q^2\right)+q\left(1+k^2 p^2\right)}{(1+k^2 p^2)( 1+k^2 q^2)}\right).
$$
Moreover, the corresponding eigenvalues of $\mathcal{K}_0(\ell)$ have opposite signs, so that any of these collisions may lead to unstable eigenvalues of the operator $\mathcal{A}_a(\ell)$.
\subsubsection{Long wavelength transverse perturbations}
\begin{lemma}\label{lem4.5}
Assume that $|\ell|$ and $|a|$ are sufficiently small. Recall the definition \eqref{def l_a} for $\ell_a^2$.
If $\ell^2_a < 0$, then there exists some
\[
\ell_{**}^2 := -\ell^2_a + O(a^4) > 0.
\]
\begin{enumerate}[label=\textup{(\roman*)}]
\item for any $\ell^2 > \ell_{**}^2$, the spectrum of $\mathcal{A}_a(\ell)$ is purely imaginary.
\item for any $\ell^2< \ell_{**}^2$, the spectrum of $\mathcal{A}_a(\ell)$ is purely imaginary, except for a pair of simple real eigenvalues with opposite signs.
\end{enumerate}
If $\ell_a^2 > 0$, then the spectrum of $\mathcal{A}_a(\ell)$ is purely imaginary.
\end{lemma}
\begin{proof} Upon replacing $\ell^2$ by $-\ell^2$ in \eqref{4.14}, we find that the two eigenvalues satisfy
$$
\lambda^2=-\frac{ \ell^2}{ (1 + k ^ { 2 })^2 }\left[ \ell^2 + \ell^2_a + O\left(a^2\left(\ell^2+a^2\right)\right) \right].
$$
Consequently, we may process as in Theorem \ref{the4.1} to get the results.
\end{proof}

The parameter regimes indicating the sign of $\ell_a^2$ is given in Lemma \ref{lem4.4}, which together with Lemma \ref{lem4.5} provides the spectral stability for the $b$-KP-II case.

%Similar as Lemma \ref{lem4.4}, we list the cases of $\ell_a>0$ and $\ell_a\leq0$ in the following lemma.
%\begin{lemma}\label{lem4.6}
%$\ell_a>0$ is valid for the following cases:
%\[
%\textup{(1)}\ b>\frac 7 2, \ k^2> \frac {b+1} {2b-7}; \qquad \textup{(2)} \ b<-1, \ k^2> \frac {b+1} {2b-7}.
%\]
%%\begin{enumerate}
%%  \item $b>\frac 7 2$, $k^2> \frac {b+1} {2b-7}$;
%%  \item $b<-1$, $k^2> \frac {b+1} {2b-7}$;
%%\end{enumerate}
%%On the other hand, $\ell_a\leq0$ is valid for the following cases:
%Or else, $\ell_a\leq0$, that is, $\ell_a\leq0$ when
%\[
%\textup{(1)} \ -1\leq b\leq\frac 7 2; \qquad \textup{(2)} \ b>\frac 7 2, \ k^2\leq\frac {b+1} {2b-7}; \qquad \textup{or \ \ (3)} \ b<-1 \ k^2\leq\frac{b+1} {2b-7}.
%\]
%%\begin{enumerate}
%%  \item $-1\leq b\leq\frac 7 2$;
%%  \item $b>\frac 7 2$, $k^2\leq\frac {b+1} {2b-7}$;
%%  \item $b<-1$, $k^2\leq\frac{b+1} {2b-7}$.
%%\end{enumerate}
%\end{lemma}

\section{Non-periodic perturbations for $b$-KP-I: onset of instability}\label{sec nonper}
In this section, we will consider the two-dimensional perturbations which are non-periodic (localized
or bounded) in the direction of the propagation of the wave. For non-periodic perturbations, we study the invertibility of $\mathcal{T}_a(\lambda, \ell)$ in \eqref{3.4} acting in $L^2(\mathbb{R})$ or $C_{\textup{bdd}}(\mathbb{R})$ (with domain $H^{4}(\mathbb{R})$ or $C_{\textup{bdd}}^{4}(\mathbb{R})$ ), for $\lambda \in \mathbb{C}, \ \Re(\lambda)>0$, and $\ell \in \mathbb{R} \backslash\{0\}$. The notable difference in this case is that $\mathcal{T}_a(\lambda, \ell)$ now has bands of continuous spectrum.

\subsection{Reformulation and main result}
Since the coefficients of $\mathcal{T}_a(\lambda, \ell)$ are periodic functions, using Floquet theory, all solutions of $\mathcal{T}_a(\lambda, \ell) V = 0$ in $L^2(\mathbb{R})$ or $C_{\textup{bdd}}(\mathbb{R})$ are of the form $V(z)=\mathrm{e}^{i \xi z} \tilde{V}(z)$ where $\xi \in\left(-\frac{1}{2}, \frac{1}{2}\right]$ is the Floquet exponent and $\tilde{V}$ is a $2 \pi$-periodic function; see \cite{H1} for a similar situation. This replaces the study of invertibility of the operator $\mathcal{T}_a(\lambda, \ell)$ in $L^2(\mathbb{R})$ or $C_{\textup{bdd}}(\mathbb{R})$ by the study of invertibility of a family of Bloch operators in $L^2(\mathbb{T})$ parameterized by the Floquet exponent $\xi$. We present the precise reformulation in the following lemma.
\begin{lemma}\label{lem5.1}
The linear operator $\mathcal{T}_a(\lambda, \ell)$ is invertible in $L^2(\mathbb{R})$ if and only if the linear operators
\begin{equation*}%\label{5.1}
\begin{aligned}
&\mathcal{T}_{a, \xi}(\lambda, \ell)\nonumber\\
&\ =\left(1-k^2 \left(\partial_z+i \xi\right)^2\right) \left(\partial_z+i \xi\right)\left\{\lambda-c\left(\partial_z+i \xi\right)
+\left(1-k^2 \left(\partial_z+i \xi\right)^2\right)^{-1}\right.\nonumber\\
&\ \quad \left.\left(\partial_z+i \xi\right)\left[ \kappa +(b+1) w-k^2w_{z z}
-(b-1)k^2w_z\left(\partial_z+i \xi\right) -k^2w\left(\partial_z+i \xi\right)^2\right] \right\}-\sigma\ell^2
\end{aligned}
\end{equation*}
acting in $L^2(\mathbb{T})$ with domain $H^4_{\textup{per}}(\mathbb{T})$ are invertible for all $\xi \in\left(-\frac{1}{2}, \frac{1}{2}\right]$.
\end{lemma}
We refer to \cite[Proposition 1.1]{H1} for a detailed proof in the similar situation. The fact that the operators $\mathcal{T}_{a, \xi}(\lambda, \ell)$ act in $L^2(\mathbb{T})$ with compactly embedded domain $H^4_{\textup{per}}(\mathbb{T})$ implies that these operators have only point spectrum. Noting that $\xi=0$ corresponds to the periodic perturbations which we have already investigated, we would restrict ourselves to the case of $\xi \neq 0$. Thus the operator $\partial_z+i \xi$ is invertible in $L^2(\mathbb{T})$. Using this, we have the following result.
\begin{lemma}\label{lem5.2}
The operator $\mathcal{T}_{a, \xi}(\lambda, \ell)$ is not invertible in $L^2(\mathbb{T})$ for some $\lambda \in \mathbb{C}$ and $\xi \neq 0$ if and only if $\lambda \in \operatorname{spec}_{L^2(\mathbb{T})}\left(\mathcal{A}_a(\ell, \xi)\right)$, where
\begin{equation*}%\label{5.2}
\begin{aligned}
&\mathcal{A}_a(\ell, \xi)\nonumber\\
=&\left(\partial_z+i \xi\right) \left[c-\left(1-k^2 \left(\partial_z+i \xi\right)^2\right)^{-1}\right.\nonumber\\
&\left.\left( \kappa +(b+1) w-k^2w_{z z}
-k^2(b-1)w_z\left(\partial_z+i \xi\right) -k^2w\left(\partial_z+i \xi\right)^2-\sigma\ell^2 \left(\partial_z+i \xi\right)^{-2}\right)\right]\nonumber\\
=&\left(\partial_z+i \xi\right)\left[ 1-k^2\left(\partial_z+i \xi\right)^2\right]^{-1} \Big[ c\left(1-(\partial_z+i \xi)^2\right) - \nonumber\\
&\left.\left( \kappa +(b+1) w-k^2w_{z z}
-k^2(b-1)w_z\left(\partial_z+i \xi\right)
-k^2w\left(\partial_z+i \xi\right)^2-\sigma\ell^2 \left(\partial_z+i \xi\right)^{-2}\right)\right].
\end{aligned}
\end{equation*}
\end{lemma}
Note that the operator $\left(\partial_z+i \xi\right)^{-1}$ becomes singular as $\xi \rightarrow 0$. Thus the implication from the spectral information of $\mathcal{A}_a(\ell, \xi)$ to the invertibility of $\mathcal{T}_{a, \xi}(\lambda, \ell)$ is not uniform in $\xi$. Therefore we will restrict our attention to the case when $|\xi|>\epsilon>0$, and look to detect the onset of instability.

%So replacing the study of the invertibility of $\mathcal{T}_{a, \xi}(\lambda, \ell)$ by the study of the spectrum of $\mathcal{A}_a(\ell, \xi)$ is not suitable for small $\xi$. To avoid this, we only study the spectrum of $\mathcal{A}_a(\ell, \xi)$ for $|\xi|>\epsilon>0$.

\begin{lemma}\label{lem-symmetry}
Assume that $\xi \in\left(-\frac{1}{2}, \frac{1}{2}\right]$ and $\xi\neq0$. Then the spectrum $\operatorname{spec}_{L^2(\mathbb{T})}\left(\mathcal{A}_a(\ell, \xi)\right)$ is symmetric with respect to the imaginary axis, and $\operatorname{spec}_{L^2(\mathbb{T})}\left(\mathcal{A}_a(\ell, \xi)\right)= \operatorname{spec}_{L^2(\mathbb{T})}\left(-\mathcal{A}_a(\ell,-\xi)\right)$.
\end{lemma}
\begin{proof} We consider $\mathcal{S}$ to be the reflection through the imaginary axis defined as follows
\[
\mathcal{S}\psi(z)=\overline{\psi(-z)},
\]
and notice that $\mathcal{A}_a(\ell, \xi)$ anti-commutes with $\mathcal{S}$,
\[
\left(\mathcal{A}_a(\ell, \xi)\mathcal{S}\psi\right)(z)=\mathcal{A}_a(\ell, \xi)\left(\overline{\psi(-z)}\right)
=-\overline{\left(\mathcal{A}_a(\ell, \xi)\psi\right)}(-z)=-\left(\mathcal{S}\mathcal{A}_a(\ell, \xi)\right)\psi(z),
\]
where we have used the fact that $w$ is an even function. Assume $\mu$ is the eigenvalue of $\mathcal{A}_a(\ell, \xi)$ with an associated eigenvector $\varphi$,
\[
\mathcal{A}_a(\ell, \xi)\varphi=\mu\varphi.
\]
then we have
\[
\mathcal{A}_a(\ell, \xi)\mathcal{S}\varphi=-\mathcal{S}\mathcal{A}_a(\ell, \xi)\varphi=-\overline{\mu}\mathcal{S}\varphi.
\]
Consequently, $-\overline{\mu}$ is an eigenvalue of $\mathcal{A}_a(\ell, \xi)$. This implies that the spectrum of $\mathcal{A}_a(\ell, \xi)$ is symmetric with respect to imaginary axis.

Consider $\mathcal{R}$ to be the reflection as follows
\[
\mathcal{R}\psi(z)=\psi(-z),
\]
then we have
\[
\left(\mathcal{A}_a(\ell, \xi)\mathcal{R}\right)\psi(z)=\mathcal{A}_a(\ell, \xi)\left(\psi(-z)\right)=-\left(\mathcal{A}_a(\ell, -\xi)\psi\right)(-z)
=-\left(\mathcal{R}\mathcal{A}_a(\ell, -\xi)\psi\right)(z).
\]
This gives the second property.
%Assume $\mu$ is the eigenvalue of $-\mathcal{A}_a(\ell, -\xi)$ with an associated eigenvector $\varphi$,
%\[
%-\mathcal{A}_a(\ell, -\xi)\varphi=\mu\varphi.
%\]
%then we have
%\[
%\mathcal{A}_a(\ell, \xi)\mathcal{R}\varphi=-\mathcal{R}\mathcal{A}_a(\ell, -\xi)\varphi=\mu\mathcal{R}\varphi.
%\]
%then we get if $\mu\in \operatorname{spec}_{L^2(\mathbb(T)}\left(\mathcal{A}_a(\ell,\xi)\right)$, then $\overline{\mu}\in \operatorname{spec}_{L^2(\mathbb(T)}\left(\mathcal{A}_a(\ell,-\xi)\right)$.
\end{proof}

From the above lemma, we can without loss of generality assume that $\xi \in\left(0, \frac{1}{2}\right]$. We will study the $L^2(\mathbb{T})$-spectra of the linear operators $\mathcal{A}_a(\ell, \xi)$ for $|a|$ sufficiently small. It is straightforward to establish the estimate
$$
\left\|\mathcal{A}_a(\ell, \xi)-\mathcal{A}_0(\ell, \xi)\right\|_{H^1(\mathbb{T})\rightarrow L^2(\mathbb{T})}=O(|a|)
$$
as $a \rightarrow 0$ uniformly for $\xi \in\left(0, \frac{1}{2}\right]$ in the operator norm. Therefore, in order to locate the spectrum of $\mathcal{A}_a(\ell, \xi)$, we need to determine the spectrum of $\mathcal{A}_0(\ell, \xi)$. A simple calculation yields that
\begin{subequations}\label{ev A nonper}
\begin{equation}\label{ev A FS}
\mathcal{A}_0(\ell, \xi) \mathrm{e}^{i n z}=i \omega_n (\ell, \xi) \mathrm{e}^{i n z}, \quad n \in \mathbb{Z},
\end{equation}
where
\begin{equation}\label{ev A FS r}
\omega_{n}(\ell, \xi) =(n+\xi)\left(\frac{\kappa}{1+k^2}-\frac{\kappa}{1+k^2 (n+\xi)^2}-\frac{\sigma\ell^2}{(n+\xi)^2\left(1+k^2 (n+\xi)^2\right)}\right).
\end{equation}
\end{subequations}

As in the previous section, the linear operator $\mathcal{A}_0(\ell, \xi)$ can be decomposed as
$$
\mathcal{A}_0(\ell,\xi)=J_{\xi} \mathcal{K}_0(\ell, \xi),
$$
where
$$
J_{\xi} := \left(\partial_z+i \xi\right)\left(1-k^2\left(\partial_z+i \xi\right)^2\right)^{-1},
$$
and
\begin{equation*}
\begin{aligned}
\mathcal{K}_0(\ell, \xi)
:= &\left(c_0\left(1-(\partial_z+i \xi)^2\right)- \kappa +\sigma\ell^2 \left(\partial_z+i \xi\right)^{-2}\right).
\end{aligned}
\end{equation*}
The operator $J_{\xi}$ is skew-adjoint, whereas the operator $\mathcal{K}_0(\ell, \xi)$ is self-adjoint. As defined in \eqref{3.16}, the Krein signature, $K_{n, \xi}$ of an eigenvalue $i \omega_n(\ell, \xi)$ in $\operatorname{spec}\left(\mathcal{A}_0(\ell, \xi)\right)$ is
$$
K_{n, \xi}=\operatorname{sgn}\left(\mu_n(\ell, \xi) := \frac { \kappa k^2\left((n+\xi)^2-1\right)}{1+k^2}-\frac{\sigma\ell^2}{(n+\xi)^2}\right)
%\kappa_{n, \xi}=\operatorname{sgn}\left(\mu_n(\ell, \xi)=\left\{ \kappa \frac {k^2\left((n+\xi)^2-1\right)}{1+k^2}-\frac{\sigma\ell^2}{(n+\xi)^2}\;;\;n\in \mathbb{Z}^*\right\}\right),
$$
for $n \in \mathbb{Z}$. Note that we have
\begin{equation}\label{omega mu}
\omega_{n}(\ell, \xi) = \frac{n + \xi}{1 + k^2(n+\xi)^2} \mu_{n}(\ell, \xi).
\end{equation}
As explained in the previous section, the $b$-KP-II case is very difficult to analyze, and hence we will mainly focus on the $b$-KP-I equation ($\sigma=-1$).

%\subsection{$b$-KP-I equation}
The main result of this section is the following theorem showing the finite-wavelength transverse spectral instability of the periodic waves for the $b$-KP-I equation under perturbations which are non-periodic in $z$.
\begin{theorem}[Instability under finite-wavelength transverse perturbation] \label{the5.1}
Consider $\sigma = -1$. Assume that $\xi \in\left(0, \frac{1}{2}\right]$ and define
\begin{equation}\label{def l_c}
\begin{split}
\ell_0^2 & :=\frac{ \kappa k^2(1-\xi^2)\xi^2}{\left(1+k^2\right)}, \\
\ell_c^2 & := \frac{\kappa k^2(1-\xi)^2\xi^2}{\left(1+k^2\right)}\cdot\frac{\left(1+\xi\right)\left[ 1+k^2(1-\xi)^2\right]
+\left(1+k^2\xi^2\right)\left(2-\xi\right)}{\left(1-\xi\right)\left[ 1+k^2(1-\xi)^2\right]
+\left(1+k^2\xi^2\right)\xi},
\end{split}
\end{equation}
and
\begin{equation}\label{def B}
B := \left[ k^2\xi^2+(1-b)k^2\xi+k^2+(b+1)\right]
\left[ k^2\xi^2+(b-3)k^2\xi+(3-b)k^2+(b+1)\right].
\end{equation}
Therefore we know that $\ell_0^2 {\leq} \ell_c^2$. {Then for $\ell^2 \ge \ell_0^2$, we have}
\begin{enumerate}[label=\textup{(\Roman*)}]
\item\label{B>0}{In the case of $B>0$, for any $|a|$ sufficiently small, there exists $\varepsilon_a(\xi)>0$ with
\begin{equation}\label{def epsilon_a}
\varepsilon_a(\xi) :=\xi^{\frac 3 2}(1-\xi)^{\frac 3 2}\frac{(1+k^2\xi^2)^\frac{1} {2}\left(1+k^2(1-\xi)^2\right)^\frac{1} {2}}{(1-\xi)\left[1+k^2(1-\xi)^2\right]+\xi(1+k^2\xi^2)}B^{\frac 1 2} |a|
\end{equation}
such that
\begin{enumerate}[label=\textup{(\roman*)}]
\item\label{bdd stab finite} for $\left|\ell^2-\ell_c^2(\xi)\right| > \varepsilon_a(\xi)$, the spectrum of $\mathcal{A}_a(\ell, \xi)$ is purely imaginary;
\item for $\left|\ell^2-\ell_c^2(\xi)\right|<\varepsilon_a(\xi)$, the spectrum of $\mathcal{A}_a(\ell, \xi)$ is purely imaginary, except for a pair of complex eigenvalues with opposite nonzero real parts.
\end{enumerate}
}
\item\label{B<0}{In the case when $B < 0$, the spectrum of $\mathcal{A}_a(\ell, \xi)$ is purely imaginary.}
\end{enumerate}
\end{theorem}
\begin{remark}
From Lemma \ref{lem5.7} we see that the transverse spectral instability holds for $-1 \le b \le 3$, which covers the well-known examples of CH-KP-I ($b = 2$) and DP-KP-I ($b = 3$).
\end{remark}
The remainder of this subsection aims at proving this theorem, and the main argument is provided in Section \ref{subsec bddinstab}.

\subsection{The Krein signature and stability under short-wavelength transverse perturbations}
%\subsection{The Krein signature $K_{n, \xi}$ and spectrum of $\mathcal{K}_0(\ell, \xi)$} {Spectrum of $\mathcal{K}_0(\ell, \xi)$ and consequences}
We start the analysis of the spectrum of $\mathcal{A}_0(\ell, \xi)$ with the values of $\ell$ away from the origin, $|\ell|>\ell_0$, for some $\ell_0 >0$. Recall that now we take $\sigma = -1$. It is straight forward to verify that
 \begin{itemize}
   \item $K_{n, \xi} = 1$  for all $n \in \mathbb{Z} \backslash\{-1,0\}$, $k > 0$, and $\xi \in\left(0, \frac{1}{2}\right]$, as
$$
\mu_n(\ell, \xi)> \frac { \kappa k^2\left((n+\xi)^2-1\right)}{1+k^2}\geqslant \frac{\kappa k^2}{\left(1+k^2\right)}\xi(2+\xi);
$$
   \item when $n = -1$, the eigenvalue
\begin{equation}
\mu_{-1}(\ell, \xi)
=\left(\frac{\ell^2}{(1-\xi)^2}-\frac{\kappa k^2\xi(2-\xi)}{\left(1+k^2\right)}\right) \nonumber
\end{equation}
is positive when $\ell^2>\ell_{-}^2$, where $\ell_{-}^2 := \frac{\kappa k^2\xi(2-\xi)(1-\xi)^2}{\left(1+k^2\right)}$; it is zero when $\ell^2=\ell_{-}^2$, and it is negative when $\ell^2<\ell_{-}^2$;
   \item when the Fourier mode $n=0$, the eigenvalue
\begin{equation}
\mu_{0}(\ell, \xi)
=\left(\frac{\ell^2}{\xi^2}-\frac{\kappa k^2(1-\xi^2)}{\left(1+k^2\right)}\right) \nonumber
\end{equation}
is positive when $\ell^2>\ell_{0}^2$, where $\ell_{0}^2$ is defined in \eqref{def l_c}; it is zero when $\ell^2=\ell_{0}^2$, and it is negative when $\ell^2<\ell_{0}^2$. (see also Figure \ref{fig FG} (a)).
 \end{itemize}
\begin{figure}[htbp]
\centering
\subfigure[]
{\begin{minipage}[t]{0.48\textwidth}
\centering
\includegraphics[width=2.95in]{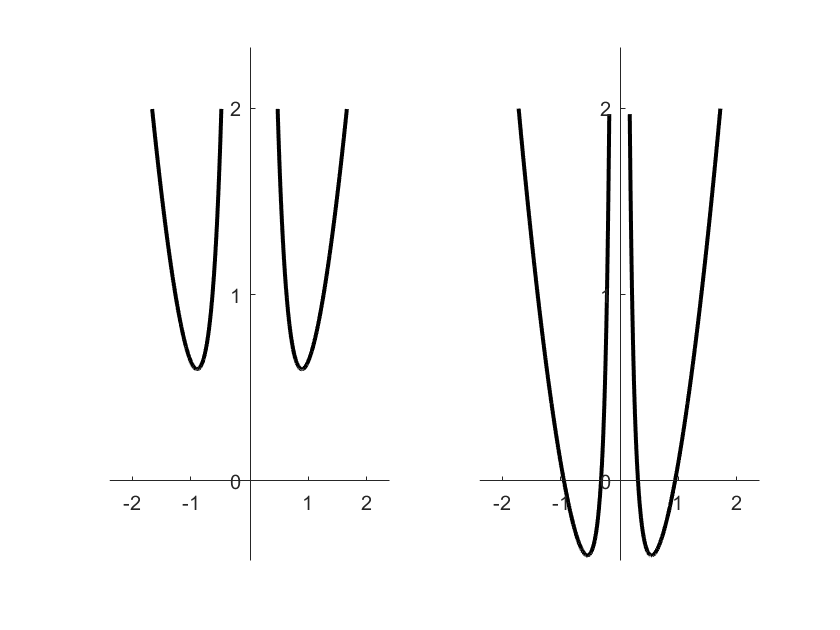}
\end{minipage}
}
\subfigure[]
{\begin{minipage}[t]{0.48\textwidth}
\centering
\includegraphics[width=2.95in]{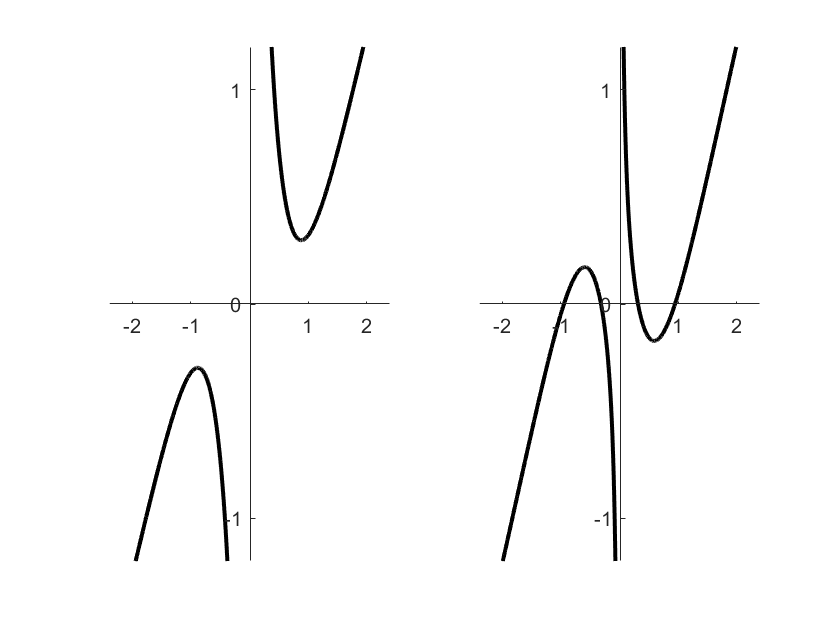}
\end{minipage}
}
\caption{[$b$-KP-I equation] (a) Graph of the map $n \mapsto \frac{ \ell^2}{n ^ { 2 }} - \frac{ \kappa k^2(1 - n^2)}{1+k^2}$ for $\kappa =2, k=1$ and $\ell=0.8$ and $\ell=0.3$ (from left to right). The eigenvalues of $\mathcal{K}_0(\ell)$ are found by taking $n=p+\xi, p \in \mathbb{Z}$. (b) Graph of the dispersion relation $n \mapsto n\left(\frac{\kappa}{1+k^2}-\frac{\kappa}{1+k^2 n^2}+ \frac{ \ell^2}{n ^ { 2 } ( 1 + k ^ { 2 } n ^ { 2 })}\right)$ for the same values of $\ell, \kappa$ and $k$. The imaginary parts of the eigenvalues of $\mathcal{A}_0(\ell)$ are found by taking $n=p+\xi, p \in \mathbb{Z}$. Notice that the zeros of the two maps are the same. \label{fig FG}}
\end{figure}

Here
$$
\ell_0^2=\frac{ \kappa k^2(1-\xi^2)\xi^2}{\left(1+k^2\right)}
<\ell_{-}^2=\frac{ \kappa k^2(1-\xi)^2\xi(2-\xi)}{\left(1+k^2\right)},
$$
for any $\xi\in\left(0, \frac{1}{2}\right)$, so that the unperturbed operator $\mathcal{K}_0(\ell, \xi)$ has positive spectrum for $\ell^2>\ell_{-}^2$, one negative eigenvalue if $\ell_{0}^2 \le \ell^2<\ell_{-}^2$, and two negative eigenvalues if $\ell^2<\ell_{0}^2$. The following result is an immediate consequence of these properties and we refer to \cite [Lemma 5.4]{Har11} for a detailed proof in a similar situation.
\begin{lemma}[Stability under short-wavelength transverse perturbation]\label{lem5.3}
Assume that $\xi \in\left(0, \frac{1}{2}\right]$. For any $\varepsilon_*>0$ there exists $a_*>0$, such that the spectrum of $\mathcal{A}_a(\ell, \xi)$ is purely imaginary, for any $\ell$ and a satisfying $\ell^2 \geqslant \ell_{-}^2+\varepsilon_*$ and $|a| \leqslant a_*$.
\end{lemma}
It remains to determine the spectrum of $\mathcal{A}_a(\ell, \xi)$ for $0 < \ell^2 < \ell_{-}^2+\varepsilon_*$. We proceed as in Section \ref{se4.1} to decompose the spectrum of $\mathcal{A}_a(\ell, \xi)$ into $\sigma_0\left(\mathcal{A}_a(\ell, \xi)\right)$ containing a minimal number of eigenvalues, and $\sigma_1\left(\mathcal{A}_a(\ell, \xi)\right)$ for which we argue as in Theorem \ref{the4.1} to show that it is purely imaginary.

\subsection{Spectral decomposition of $\mathcal{A}_0(\ell, \xi)$ for $\ell_0^2 \le \ell^2 < \ell_{-}^2+\varepsilon_*$}\label{subsec spec decomp} % and decomposition}
Recall that the spectrum of $\mathcal{A}_0(\ell, \xi)$ is given in \eqref{ev A nonper}. The distribution of these eigenvalues on the imaginary axis can be inferred from the study of the dispersion relation (see Figure \ref{fig FG} (b)). The discussion preceding Lemma \ref{lem5.3} implies that the eigenvalues of $\mathcal{A}_0(\ell, \xi)$ that might lead to instability under perturbations are
\[
\begin{aligned}
& i \omega_{-1}(\ell,\xi), \quad \qquad & \text{if } & \ \ell_0^2 \le \ell^2 {\leq} \ell_-^2, \\
& i \omega_{-1}(\ell,\xi),\ i \omega_{0}(\ell,\xi), & \text{if } & \ 0 < \ell^2 < \ell_0^2.
\end{aligned}
\]
The next step is to separate these eigenvalues from the remaining point spectra, which correspond to the ones with positive Krein signature. Such a separation is possible as long as there are no collisions between these eigenvalues and other ones. The following lemma confirms this lack of collision for transverse perturbations in the finite wave length regime.
\begin{lemma}\label{lem collision}
{Assume that $\xi \in\left(0, \frac{1}{2}\right]$.} The eigenvalues of $\mathcal{A}_0(\ell, \xi)$ satisfy
\begin{subequations}\label{eq collision}
\begin{equation}\label{eq collision 1}
\begin{split}
\omega_{-1}(\ell, \xi) & = \omega_{0}(\ell, \xi) \quad \text{only when }\ \ell^2 = \ell_c^2, \\
\omega_{-1}(\ell, \xi) & \ne \omega_{n}(\ell, \xi) \quad \text{for all }\ n \ne 0, -1, 1.
\end{split}
\end{equation}
When $\ell_0^2 \le \ell^2 {\leq} \ell_-^2$ it further holds that
\begin{equation}\label{eq collision 2}
\omega_{-1}(\ell, \xi) \ne \omega_{n}(\ell, \xi) \quad \text{for all }\ n \ne 0, -1.
\end{equation}
\end{subequations}
\end{lemma}
\begin{proof}
Define the collision function
\begin{align*}
F_n(\ell^2, \xi) := & \ \omega_n(\ell,\xi) - \omega_{-1}(\ell, \xi) = \frac{1 - \xi}{1 + k^2(1-\xi)^2} \mu_{-1}(\ell,\xi) + \frac{n + \xi}{1 + k^2 (n + \xi)^2} \mu_n(\ell, \xi) \\
= & \ \frac{\kappa k^2}{1 + k^2} \LB \frac{(1 - \xi) \LB (1 - \xi)^2 -1 \RB}{1 + k^2 (1 - \xi)^2} + \frac{(n + \xi) \LB (n + \xi)^2 - 1 \RB}{1 + k^2 (n + \xi)^2} \RB + \\
& \ \ell^2 \LB \frac{1}{(1 - \xi) \LB 1 + k^2 (1 - \xi)^2 \RB} + \frac{1}{(n + \xi) \LB 1 + k^2(n + \xi)^2 \RB} \RB \\
=: & \ G_n(\xi) + \ell^2 H_n(\xi).
\end{align*}
Clearly $F_n$ is linear in $\ell^2$. From the previous discussion we see that
\[
\omega_{-1}(\ell,\xi) {\geq} 0, \qquad \text{and} \qquad \omega_{n}(\ell, \xi)
\left\{\begin{aligned}
& \ge 0, \quad \text{for} \;n\geqslant0, \\														
& <0,\quad \text{for} \;n\leq-2. \end{aligned} \right.
\]
Thus only $\omega_{n}(\ell, \xi)$ with $n \ge 0$ is possible to collide with $\omega_{-1}(\ell, \xi)$.

When $n = 0$, we consider solving
\[
F_0(\ell^2) = \frac{1 - \xi}{1 + k^2(1-\xi)^2} \mu_{-1}(\ell,\xi) + \frac{\xi}{1 + k^2 \xi} \mu_0(\ell, \xi) = 0
\]
for $\ell^2$. Since both $\mu_{-1}(\ell,\xi)$ and $\mu_{0}(\ell,\xi)$ are linear increasing functions in $\ell^2$, so is $F_0(\ell^2, \xi)$. Also we have $F_0(\ell_0^2, \xi) < 0$ and $F_0(\ell_-^2, \xi) > 0$ {for $\xi \in\left(0, \frac{1}{2}\right)$}. Hence for any given {$\xi \in \left(0, \frac{1}{2}\right)$} there exists a unique $\ell_c^2 \in (\ell_0^2, \ell_-^2)$ such that $F(\ell_c^2,\xi) = 0$. Explicit computation reveals that $\ell_c^2$ is given in \eqref{def l_c}. {For the case $\xi=\frac{1}{2}$, we have $\ell_c^2=\ell_0^2=\ell_-^2$, and hence $\omega_{-1}(\ell_c^2,\xi)= \omega_{0}(\ell_c^2,\xi)=0$.} This proves the first part of \eqref{eq collision}.

Now for $n \ge 1$, the function $G_n(\xi)$ is related to the function
\[
f(x) = \frac{(1 - \xi) \LB (1 - \xi)^2 -1 \RB}{1 + k^2 (1 - \xi)^2} + \frac{(x + \xi) \LB (x + \xi)^2 - 1 \RB}{1 + k^2 (x + \xi)^2},
\]
which is increasing for $|x + \xi| > \frac{1}{\sqrt3}$. Moreover,
\[
f(2) = \frac {\left(6+9\xi+9\xi^2\right)+3k^2\left(1-\xi-\xi^2\right)\left(2+\xi\right)\left(1-\xi\right)}
{\left(1+k^2(2+\xi)^2\right)\left(1+k^2(1-\xi)^2\right)}>0.
\]
Therefore $\omega_{-1}(\ell, \xi)$ does not collide with $\omega_{n}(\ell, \xi)$ for $n \ge 2$, and hence it suffices to consider the case when $n = 1$.
Direct computation shows that
\[
G_1(\xi) = \frac{\kappa k^2}{1 + k^2}  \frac{2 \xi^2 \LB 3 - k^2 (1 - \xi^2) \RB}{\LB 1 + k^2 (1 - \xi)^2 \RB \LB 1 + k^2 (1 + \xi)^2 \RB},
\]
and
\[
G_1(\xi) < 0 \quad \text{only if} \quad k^2 > \frac{3}{1 - \xi^2}.
\]
In this case, solving $F_1(\ell^2, \xi) = 0$ we find that the unique solution for $\ell^2$ is
\[
\ell^2 = \frac{ \kappa k^2(1-\xi^2)\xi^2}{\left(1+k^2\right)} \cdot\frac{k^2(1-\xi^2)-3}{1+k^2\left(1+3\xi^2\right)} < \frac{ \kappa k^2(1-\xi^2)\xi^2}{\left(1+k^2\right)} = \ell_0^2.
\]
Therefore the second part of \eqref{eq collision} is proved.
\end{proof}

From the above lemma it follows that
\begin{lemma}\label{lem5.4}
Given $\xi \in\left(0, \frac{1}{2}\right]$, there exist $\varepsilon_*>0$ and $c_*>0$ such that
\begin{enumerate}[label=\textup{(\roman*)}]
\item for any $\ell$ satisfying $\ell_c^2+\varepsilon_*<\ell^2<\ell_{-}^2+\varepsilon_*$, the spectrum of $\mathcal{A}_0(\ell, \xi)$ decomposes as
$$
\sigma\left(\mathcal{A}_0(\ell, \xi)\right) =\left\{i \omega_{-1}(\ell, \xi)\right\} \cup \sigma_1\left(\mathcal{A}_0(\ell, \xi)\right),
$$
with $\operatorname{dist}\left(i \omega_{-1}(\ell, \xi), \sigma_1\left(\mathcal{A}_0(\ell, \xi)\right)\right) \geqslant c_*  >0$;
\item for any $\ell$ satisfying $\ell_0^2 \le \ell^2 \leqslant \ell_c^2+\varepsilon_*$, the spectrum of $\mathcal{A}_0(\ell, \xi)$ decomposes as
$$
\sigma\left(\mathcal{A}_0(\ell, \xi)\right)=\left\{i \omega_{-1}(\ell, \xi), i \omega_0(\ell, \xi)\right\} \cup \sigma_1\left(\mathcal{A}_0(\ell, \xi)\right),
$$
with $\operatorname{dist}\left(\left\{i \omega_{-1}(\ell, \xi), i \omega_0(\ell, \xi)\right\}, \sigma_1\left(\mathcal{A}_0(\ell, \xi)\right)\right) \geqslant c_*>0$.
\end{enumerate}
\end{lemma}
Continuity arguments show that for sufficiently small $a$, this decomposition persists for the operator $\mathcal{A}_a(\ell, \xi)$, and we argue as in Section \ref{4.1.2} to locate the spectrum of $\mathcal{A}_a(\ell, \xi)$.

\subsection{Spectrum of $\mathcal{A}_a(\ell, \xi)$ for $\ell_0^2 \le \ell^2<\ell_{-}^2+\varepsilon_*$}\label{subsec bddinstab}
Let us start with the case $\ell_c^2+\varepsilon_*<\ell^2<\ell_{-}^2+\varepsilon_*$. The following lemma asserts that for $\ell$ in this range the spectrum of $\mathcal{A}_a(\ell, \xi)$ is purely imaginary. The argument follows in a  similar way as that of \cite[Lemma 5.6]{Har11}, and hence we omit the proof here.
\begin{lemma} \label{lem5.5}
Assume that $\xi \in\left(0, \frac{1}{2}\right]$. There exist $\varepsilon_*>0$ and $a_*>0$ such that the spectrum of $\mathcal{A}_a(\ell, \xi)$ is purely imaginary, for any $\ell$ and a satisfying $\ell_c^2+\varepsilon_*<\ell^2<\ell_{-}^2+\varepsilon_*$ and $|a| \leqslant a_*$.
\end{lemma}

Next we consider the spectrum of $\mathcal{A}_a(\ell, \xi)$ for $\ell_0^2 \le \ell^2 \leqslant \ell_c^2+\varepsilon_*$. From Section \ref{subsec spec decomp} we know that the two eigenvalues of $\mathcal{A}_0(\ell, \xi)$ corresponding to the Fourier modes $n=-1$ and $n=0$ collide at $\ell^2 = \ell_c^2$. Using perturbation arguments we prove in the following that for $a$ sufficiently small and $\ell$ close to the value $\ell_c$, the linearized operator $\mathcal{A}_a(\ell, \xi)$ will continue to accommodate a pair of unstable eigenvalues. This, together with the definition of instability, proves Theorem \ref{the5.1}.

%We prove that there is a pair of unstable eigenvalues of $\mathcal{A}_a(\ell, \xi)$, for $\ell$ close to the value $\ell_c$ at which the two eigenvalues of $\mathcal{A}_0(\ell, \xi)$ corresponding to the Fourier modes $n=-1$ and $n=0$ collide.
\begin{lemma}\label{lem5.6}
Assume that $\xi \in\left(0, \frac{1}{2}\right]$ and recall $B$ in \eqref{def B}.
%{ {
%\[
%B=\left(k^2\xi^2+(1-b)k^2\xi+k^2+(b+1)\right)
%\left(k^2\xi^2+(b-3)k^2\xi+(3-b)k^2+(b+1)\right).
%\]
%}}
In the case of $B>0$, there exist constants $\varepsilon_*, a_* > 0$, and $\varepsilon_a(\xi)$ as defined in \eqref{def epsilon_a} with $\varepsilon_a(\xi) < \varepsilon_*$, such that {if $\ell_0^2 \le \ell^2 \leqslant \ell_c^2+\varepsilon_*$} and
\begin{enumerate}[label=\textup{(\roman*)}]
\item $\left|\ell^2-\ell_c^2\right| > \varepsilon_a(\xi)$, then the spectrum of $\mathcal{A}_a(\ell, \xi)$ is purely imaginary;
\item $\left|\ell^2-\ell_c^2\right|<\varepsilon_a(\xi)$ then the spectrum of $\mathcal{A}_a(\ell, \xi)$ is purely imaginary, except for a pair of complex eigenvalues with opposite nonzero real parts.
\end{enumerate}
In the case $B < 0$, the spectrum of $\mathcal{A}_a(\ell, \xi)$ is purely imaginary for $\ell_0^2 \le \ell^2 \leqslant \ell_c^2+\varepsilon_*$.
\end{lemma}
\begin{proof}The spectrum of $\mathcal{A}_a(\ell, \xi)$ can be decomposed as
$$
\sigma\left(\mathcal{A}_a(\ell, \xi)\right)=\sigma_0\left(\mathcal{A}_a(\ell, \xi)\right) \cup \sigma_1\left(\mathcal{A}_a(\ell, \xi)\right),
$$
where $\sigma_1\left(\mathcal{A}_a(\ell, \xi)\right)$ is purely imaginary, and $\sigma_0\left(\mathcal{A}_a(\ell, \xi)\right)$ consists of two eigenvalues which are the continuation of the eigenvalues $i \omega_{-1}(\ell, \xi)$ and $i \omega_0(\ell, \xi)$ for small $a$. Choosing $\varepsilon_*$ and $a_*$ sufficiently small, such decomposition persists for any $\ell_0^2 \le \ell^2 \leqslant \ell_c^2+\varepsilon_*$ and $|a| \leqslant a_*$. Therefore what remains to check are the location of the two eigenvalues in $\sigma_0\left(\mathcal{A}_a(\ell, \xi)\right)$.

From Lemma \ref{lem collision} we know that for any $\ell_0^2 \le \ell^2 \leqslant \ell_c^2+\varepsilon_*$ such that $\ell$ outside a neighborhood of $\ell_c$, the two eigenvalues $i \omega_{-1}(\ell, \xi)$ and $i \omega_0(\ell, \xi)$ are simple and there exists $c_0 > 0$ such that
$$
\left|i \omega_{-1}(\ell, \xi)-i \omega_0(\ell, \xi)\right| \geqslant c_0.
$$
For $a$ sufficiently small, the simplicity of this pair of eigenvalues continues to hold into the spectrum of $\mathcal{A}_a(\ell, \xi)$.
%A standard perturbation argument then shows that the continuation of this eigenvalues for sufficiently small $a$ is a pair of simple eigenvalues of $\mathcal{A}_a(\ell, \xi)$.
As the spectrum of $\mathcal{A}_a(\ell, \xi)$ is symmetric with respect to the imaginary axis, each of these eigenvalues of $\mathcal{A}_a(\ell, \xi)$ is purely imaginary for any $\ell$ outside some neighborhood of $\ell_c$.

We will proceed as in the proof of Theorem \ref{the4.1} to locate $\sigma_0\left(\mathcal{A}_a(\ell, \xi)\right)$ for $\ell$ close to $\ell_c$. We compute successively a basis for the two-dimensional spectral subspace associated with $\sigma_0\left(\mathcal{A}_a(\ell, \xi)\right)$, the $2 \times 2$ matrix $M_a(\ell, \xi)$ representing the action of $\mathcal{A}_a(\ell, \xi)$ on this basis, and the eigenvalues of this matrix.

At $a=0$, the basis vectors are chosen to be the two eigenvectors associated with the eigenvalues $i \omega_0(\ell, \xi)$ and $i \omega_{-1}(\ell, \xi)$,
$$
\xi_0^0(\ell, \xi)=1, \quad \xi_0^1(\ell, \xi)=e^{-i z}
$$
At order $a$, we take $\ell=\ell_c$, and proceed as the computation in the proof of Theorem \ref{the4.1} to find
\[
\begin{aligned}
M_a\left(\ell_c, \xi\right)=&\left(\begin{array}{cc}
i \omega_0\left(\ell_c, \xi\right) & -\frac{i}{2} \frac {(\xi-1)\left(k^2\xi^2+(1-b)k^2\xi+k^2+(b+1)\right)}{1+k^2(\xi-1)^2} a \\
-\frac{i}{2}\frac {\xi\left(k^2\xi^2+(b-3)k^2\xi+(3-b)k^2+(b+1)\right)}{1+k^2\xi^2} a & i \omega_{-1}\left(\ell_c, \xi\right)
\end{array}\right)\\
&+O\left(a^2\right).
\end{aligned}
\]
Together with the expression of $M_0(\ell,\xi)$, this yields that
\[
\begin{aligned}
M_a\left(\ell, \xi\right)=&\left(\begin{array}{cc}
i \omega_0\left(\ell_c, \xi\right)+i\frac{\varepsilon}{\xi(1+k^2\xi^2)} & -\frac{i}{2} \frac {(\xi-1)\left[k^2\xi^2+(1-b)k^2\xi+k^2+(b+1)\right]}{1+k^2(\xi-1)^2} a  \\\\
-\frac{i}{2}\frac {\xi\left[ k^2\xi^2+(b-3)k^2\xi+(3-b)k^2+(b+1)\right]}{1+k^2\xi^2} a & i \omega_{-1}\left(\ell_c, \xi\right)+i\frac{\varepsilon}{(\xi-1)\left[1+k^2(\xi-1)^2\right]}
\end{array}\right)\\
&+O\left(a^2 + |a\varepsilon| \right),
\end{aligned}
\]
with $\varepsilon=\ell^2-\ell_c^2$. Recall the definition of $\ell_c^2$ in \eqref{def l_c},
%\[
%\ell_c^2=\frac{ \kappa k^2(1-\xi)^2\xi^2}{\left(1+k^2\right)}\cdot\frac{\left(1+\xi\right)\left(1+k^2(1-\xi)^2\right)
%+\left(1+k^2\xi^2\right)\left(2-\xi\right)}{\left(1-\xi\right)\left(1+k^2(1-\xi)^2\right)
%+\left(1+k^2\xi^2\right)\xi},
%\]
and the colliding eigenvalues
\[
\omega_0\left(\ell_c, \xi\right)=\omega_{-1}\left(\ell_c, \xi\right)
=\frac{2 \kappa k^2\xi(1-\xi)(1-2\xi)}{\left(1+k^2\right)\left\{ \left(1-\xi\right)\left[ 1+k^2(1-\xi)^2\right]
+\left(1+k^2\xi^2\right)\xi\right\} } =: \omega_*(\xi).
\]
Seeking eigenvalues $\lambda$ of the form
$$
\lambda = i \omega_*(\xi) + iX,
%\lambda=i\frac{2 \kappa k^2\xi(1-\xi)(1-2\xi)}{\left(1+k^2\right)\left[\left(1-\xi\right)\left(1+k^2(1-\xi)^2\right)
%+\left(1+k^2\xi^2\right)\xi\right]}+i X,
$$
we find that $X$ is root of the polynomial
\begin{align}
P(X)=&X^2-X\left(\frac{\varepsilon}{\xi(1+k^2\xi^2)}+\frac{\varepsilon}
{(\xi-1)\left(1+k^2(\xi-1)^2\right)}+O\left(a^2 + |a\varepsilon| \right)\right)\nonumber\\
&-\frac{a^2}{4} \left(\frac {(\xi-1)\left(k^2\xi^2+(1-b)k^2\xi+k^2+(b+1)\right)}{1+k^2(\xi-1)^2}\right)\nonumber\\
&\quad \ \cdot \left(\frac {\xi\left(k^2\xi^2+(b-3)k^2\xi+(3-b)k^2+(b+1)\right)}{1+k^2\xi^2}\right) \nonumber\\
&+\frac{\varepsilon^2}{\xi(\xi-1)(1+k^2\xi^2)\left(1+k^2(\xi-1)^2\right)}
+O\left(a^2 |\varepsilon| + |a| \varepsilon^2 + |a|^3 \right).\nonumber
\end{align}
A direct computation shows that the discriminant of this polynomial $P(X)$ is
\begin{equation*}
\begin{split}
\Delta_a(\varepsilon, \xi) = & \ \varepsilon^2 \LB \frac{1}{\xi(1 + k^2 \xi^2)} + \frac{1}{(1-\xi) \LB 1 + k^2(1-\xi)^2 \RB} \RB^2 \\
& \ - \frac{\xi (1 - \xi) B}{(1 + k^2 \xi^2) \LB 1 + k^2(1-\xi)^2 \RB} a^2 + O \LC |\varepsilon| a^2 + \varepsilon^2 |a| + |a|^3 \RC,
\end{split}
\end{equation*}
where $B$ is defined in \eqref{def B}.

%\begin{align}
%\Delta_a(\varepsilon, \xi)
%&=\left(\varepsilon\frac{\xi(1+k^2\xi^2)+(1-\xi)\left(1+k^2(\xi-1)^2\right)}
%{\xi(\xi-1)(1+k^2\xi^2)\left(1+k^2(\xi-1)^2\right)}\right)^2\nonumber\\
%&-\left(\frac {(1-\xi)\left(k^2\xi^2+(1-b)k^2\xi+k^2+(b+1)\right)}{1+k^2(\xi-1)^2}\right)\nonumber\\ & \quad \cdot \left(\frac {\xi\left(k^2\xi^2+(b-3)k^2\xi+(3-b)k^2+(b+1)\right)}{1+k^2\xi^2}\right)a^2\nonumber\\
%&+O\left(a^2\left(|\varepsilon|+a^2\right)\right).\nonumber
%\end{align}
In the case when $B>0$, we can define
\[
\varepsilon_a(\xi) := \xi^{\frac 3 2}(1-\xi)^{\frac 3 2}\frac{(1+k^2\xi^2)^\frac{1} {2}\left(1+k^2(1-\xi)^2\right)^\frac{1} {2}}{(1-\xi)\left[ 1+k^2(1-\xi)^2\right]+\xi(1+k^2\xi^2)}B^{\frac 1 2} \cdot |a| > 0.
\]
Therefore for any $a$ sufficiently small we have $\Delta_a(\varepsilon, \xi) \ge 0$ when $|\varepsilon| > \varepsilon_a(\xi)$ and $\Delta_a(\varepsilon, \xi) < 0$ when $|\varepsilon| < \varepsilon_a(\xi)$. This implies that, for $\ell_0^2 \le \ell^2 \le \ell_c^2 + \varepsilon_*$, the two eigenvalues of $\mathcal{A}_a(\ell, \xi)$ are purely imaginary when $\left|\ell^2-\ell_c^2\right| > \varepsilon_a(\xi)$, and complex, with opposite nonzero real parts when $\left|\ell^2-\ell_c^2\right|<\varepsilon_a(\xi)$.

In the case when $B < 0$, we have $\Delta_a(\varepsilon, \gamma)>0$, and this implies that the two eigenvalues of $\mathcal{A}_a(\ell, \xi)$ are purely imaginary for $\ell_0^2 \le \ell^2 \le \ell_c^2 + \varepsilon_*$.
\end{proof}
For $\xi\in(0,\frac 1 2]$, we list the cases of $B>0$ and $B < 0$ in the following lemma and the detailed discussion is given in the Appendix \ref{app B}.
\begin{lemma}\label{lem5.7}
$B>0$ is valid for the following cases:
\begin{enumerate}
  \item $-1\leq b\leq 3$;
  \item $b<-1$, $\xi=\frac 1 2$, $k^2\neq \frac {-4(1+b)}{7-2b}$;
  \item $b<-1$, $\xi\in(0,\frac 1 2)$, $k^2>\frac {-(1+b)}{\xi^2+(1-b)\xi+1}$ or $k^2<\frac {-(1+b)}{\xi^2+(b-3)\xi+(3-b)}$;
  \item $b>3$, $k^2\leq\frac {4} {b-3}$;
  \item $3<b\leq \frac 7 2$, $\xi=\frac 1 2$, $k^2> \frac {4} {b-3}$;
  \item $b> \frac 7 2$, $\xi=\frac 1 2$, $k^2>\frac {4} {b-3}$ and $k^2\neq\frac {-4(1+b)}{7-2b}$;
  \item $3<b\leq \frac {\xi^2-3\xi+3} {1-\xi}$, $\xi\in(0,\frac 1 2)$;
  \item $3<\frac {\xi^2-3\xi+3} {1-\xi}<b\leq\frac {\xi^2+\xi+1} {\xi}$, $\xi\in(0,\frac 1 2)$,  $\frac {4} {b-3}<k^2<\frac {-(1+b)}{\xi^2+(b-3)\xi+(3-b)}$;
  \item $b>\frac {\xi^2+\xi+1} {\xi}>3$, $\xi\in(0,\frac 1 2)$, $k^2>\frac {-(1+b)}{\left(\xi^2+(1-b)\xi+1\right)}$ or $\frac {4} {b-3}<k^2<\frac{-(1+b)}{\xi^2+(b-3)\xi+(3-b)}$.
\end{enumerate}

$B < 0$ when one of the following cases occurs:
\begin{enumerate}
%  \item $b<-1$, $\xi=\frac 1 2$, $k^2=\frac {-4(1+b)}{7-2b}$;
  \item $b<-1$, $\xi\in(0,\frac 1 2)$, $\frac {-(1+b)}{\xi^2+(1-b)\xi+1} < k^2 < \frac {-(1+b)}{\xi^2+(b-3)\xi+(3-b)}$;
%  \item $b> \frac 7 2$, $\xi=\frac 1 2$, $k^2=\frac {-4(1+b)}{7-2b}>\frac {4} {b-3}$;
  \item $3<\frac {\xi^2-3\xi+3} {1-\xi}<b\leq\frac {\xi^2+\xi+1} {\xi}$, $\xi\in(0,\frac 1 2)$,  $k^2 > \frac {-(1+b)}{\xi^2+(b-3)\xi+(3-b)}$;
  \item $b>\frac {\xi^2+\xi+1} {\xi}>3$, $\xi\in(0,\frac 1 2)$, $\frac {-(1+b)}{\xi^2+(b-3)\xi+(3-b)} < k^2 < \frac {-(1+b)}{\left(\xi^2+(1-b)\xi+1\right)}$.
\end{enumerate}
\end{lemma}

\subsection{Discussion for long-wavelength transverse perturbations $0 < \ell^2 < \ell_0^2$}\label{subsec long}

In contrast to case when {$\ell^2 \ge \ell_0^2$}, for long-wavelength transverse perturbations $0 < \ell^2 < \ell_0^2$, the collision dynamics of the eigenvalues become much harder to track. In fact, it is possible that infinitely many pairs of eigenvalues collide with each other. What we find is that, in order to eliminate these collisions, an additional condition on the longitudinal wavelength is needed. The detailed discussion is provided below.

The collision between $\omega_{-1}$ and $\omega_n$ is studied in Lemma \ref{lem collision}. So we consider the collision between $\omega_0$ and $\omega_n$ described by the function
\begin{equation}\label{eq coll func}
\begin{split}
\tilde F_n(\ell^2, \xi) := & \ \omega_n(\ell,\xi) - \omega_{0}(\ell, \xi) \\
= & \ \frac{\kappa k^2}{1 + k^2} \LB \frac{\xi \LC 1 - \xi^2 \RC}{1 + k^2 \xi^2} + \frac{(n + \xi) \LB (n + \xi)^2 - 1 \RB}{1 + k^2 (n + \xi)^2} \RB + \\
& \ \ell^2 \LB -\frac{1}{\xi \LC 1 + k^2 \xi^2 \RC} + \frac{1}{(n + \xi) \LB 1 + k^2(n + \xi)^2 \RB} \RB \\
%= & \frac{\ell_0^2 - \ell^2}{\xi \LC 1 + k^2 \xi^2 \RC} + \frac{a (n + \xi)^2 \LB (n + \xi)^2 - 1 \RB + \ell^2}{(n + \xi) \LB 1 + k^2(n + \xi)^2 \RB} %\\
=: & \ \tilde G_n(\xi) + \ell^2 \tilde H_n(\xi).
\end{split}
\end{equation}
We can also write $\tilde F_n(\ell^2, \xi)$ as
\begin{equation}\label{eq coll func another}
\tilde F_n(\ell^2, \xi) = \frac{\ell_0^2 - \ell^2}{\xi \LC 1 + k^2 \xi^2 \RC} + \frac{{\alpha} (n + \xi)^2 \LB (n + \xi)^2 - 1 \RB + \ell^2}{(n + \xi) \LB 1 + k^2(n + \xi)^2 \RB},
\end{equation}
where ${\alpha} := \frac{\kappa k^2}{1 + k^2}$. The first term in the right-hand side of \eqref{eq coll func another} is positive in the range of $\ell$ considered here. For $n \ge 1$, the second term is clearly positive. So we consider the case when $n \le -2$. Define the function
\[
g(x) := \frac{\ell^2 + {\alpha} x^2(x^2 - 1)}{x(1+k^2 x^2)} \qquad \text{for} \quad x \le -\frac32.
\]
Then from the estimate that
\[
\ell^2 < \ell_0^2 = {\alpha} \xi^2 (1 - \xi^2), \quad \text{and} \qquad \xi^2 (1 - \xi^2) < \frac14,
\]
we have
\begin{align*}
g'(x) = & \ \frac{{\alpha} \LC k^2x^6 + k^2x^4 + 3x^4 - x^2 \RC - \ell^2(3k^2 x^2 + 1)}{x^2(1+k^2 x^2)^2} \\
\ge & \ \frac{{\alpha} \LB k^2x^6 + k^2x^4 + 3x^4 - x^2 - \xi^2 (1 - \xi^2)(3k^2 x^2 + 1) \RB}{x^2(1+k^2 x^2)^2} > 0, \quad \text{for} \quad x \le -\frac32.
\end{align*}
Therefore we see that the second term in the last equality of \eqref{eq coll func another} is increasing in $n$ for $n \le -2$, that is, $\tilde F_n(\ell^2,\xi) \le \tilde F_{-2}(\ell^2,\xi)$ for $n\le -2$.

Looking at \eqref{eq coll func} we find the $\tilde H_n(\xi) < 0$ when $n \le -2$. Explicit computation reveals that
\[
\tilde G_{-2}(\xi) = \frac{2 \alpha (1 - \xi)^2 \LB k^2 \xi (2 - \xi) - 3 \RB}{(1 + k^2\xi^2) \LB 1 + k^2(\xi - 2)^2 \RB}.
%\tilde G_{-2}(\xi) = \frac{\kappa k^2}{1 + k^2} \cdot \frac{2(1 - \xi)^2 \LB k^2 \xi (2 - \xi) - 3 \RB}{(1 + k^2\xi^2) \LB 1 + k^2(\xi - 2)^2 \RB}.
\]
In fact we have
\[
\frac12 \tilde F_{-2}(\ell^2,\xi) = \frac{\LC 1 + k^2 \LB 1 + 3(1 - \xi)^2 \RB \RC \ell^2 - \LB k^2 \xi (2 - \xi) - 3 \RB \ell_-^2 }{\xi (\xi - 2)(1 + k^2\xi^2) \LB 1 + k^2(\xi - 2)^2 \RB}.
\]
Therefore we have
\[
k^2 \xi (2 - \xi) - 3 \le 0 \quad \Rightarrow \quad \tilde F_{-2}(\ell^2,\xi) < 0 \quad \Rightarrow \quad \tilde F_{n}(\ell^2,\xi) < 0 \text{ for all } n \le -2.
\]
Thus we find that for
\begin{equation}\label{eq no coll}
k^2 \le \frac{3}{\xi (2 - \xi)}, \quad \Rightarrow \quad \omega_{0}(\ell, \xi) \ne \omega_{n}(\ell, \xi) \ \text{ for all } n \ne 0, -1.
\end{equation}

To summarize, we have the following
\begin{lemma}\label{lem coll longwave}
{Assume that $\xi \in\left(0, \frac{1}{2}\right]$ and $0 < \ell^2 < \ell_0^2$.} For $k^2 \le 4$ we have $\omega_{0}(\ell, \xi) \ne \omega_{n}(\ell, \xi)$ for all $n \ne 0$ and {for $k^2 \le 3$ we have $\omega_{-1}(\ell, \xi) \ne \omega_{n}(\ell, \xi)$ for all $n \ne -1$}.
\end{lemma}
\begin{proof}
The above discussion leads to \eqref{eq no coll}. Since $0 < \xi \le \frac12$, we know that $\frac{3}{\xi (2 - \xi)} \ge 4$ {and $\frac{3}{1 - \xi^2}> 3$}. Moreover, the calculation in Lemma \ref{lem collision} indicates that $\omega_{0}(\ell, \xi) = \omega_{-1}(\ell, \xi)$ only for $\ell^2 = \ell_c^2 \ge \ell_0^2$ {and for $0 < \ell^2 < \ell_0^2$ and $k^2\leq \frac{3}{1 - \xi^2}$, $\omega_{-1}(\ell, \xi) \ne \omega_{n}(\ell, \xi)$ for all $n \ne -1$}. Thus the proof of the lemma is complete.
\end{proof}

From this lemma we may extend the stability part of Theorem \ref{the5.1} to the regime of long-wavelength transverse perturbation, provided that the longitudinal wavelength is bounded below.
\begin{proposition}\label{prop bdd stab}
If {$k^2 \leq 3$}, then the results in \ref{B>0} and \ref{B<0} parts of Theorem \ref{the5.1} hold for all $\ell^2>0$.
%\begin{enumerate}
%\item[\textup{(iii)}] \textcolor{blue}{In the case $B>0$, for} $\left|\ell^2-\ell_c^2(\xi)\right| > \varepsilon_a(\xi)$, the spectrum of $\mathcal{A}_a(\ell, \xi)$ is purely imaginary.
%\textcolor{blue}{\item[\textup{(iv)}] In the case $B < 0$, the spectrum of $\mathcal{A}_a(\ell, \xi)$ is purely imaginary.}
%\end{enumerate}
\end{proposition}

%It is easily seen that for all $0 < \xi \le \frac12$, $\tilde H_n(\xi) < 0$ when $n \le -2$.
%Explicit computation yields that $\tilde F_n(\ell^2,\xi) = 0$ at
%\begin{equation*}
%\begin{split}
%\ell^2 = & \ \frac{\kappa k^2 }{1+k^2} g(n, k, \xi),
%\end{split}
%\end{equation*}
%where
%\[
%g(n, k, \xi) := \frac{\xi (n+\xi)}{1+k^2(n^2+3n\xi+3\xi^2)} \Big[ n^2+ 3\xi^2 + 3\xi-1 +k^2\xi(n+\xi)(\xi^2+n\xi+1) \Big].
%\]
%Note that for $0 < \xi \le \frac12$,
%\[
%\frac{\xi (n+\xi)}{1+k^2(n^2+3\xi+3n\xi^2)} < 0, \ n^2 + 3 \xi^2 +3\xi-1 > 0, \text{ and } \xi(n+\xi) < 0 \qquad \text{when} \quad n \le -2.
%\]
%Thus for each $n \le -2$, when $\xi$ is small enough, say $\xi < -\frac1n$, we have $\xi^2+n\xi+1 > 0$, and hence
%\[
%g(n, k, \xi) > 0 \qquad \text{for all }\quad n \le -2, \ 0 < \xi < -\frac1n, \  k > 0.
%\]
%
%Notice that for fixed $n \le -2$, $0 < \xi < -\frac1n$, sending $k \to \infty$ we find that
%\[
%g(n, k, \xi)  \to \frac{\xi^2 (n+\xi)^2 (\xi^2+n\xi+1)}{n^2+3n\xi+3\xi^2} < \xi^2 (1 - \xi^2).
%\]

\appendix

\section{Small-amplitude expansion}\label{app A}
In this section, we give the details on  small-amplitude expansion of \eqref{2.1}. Since $w$ and $c$ depend analytically on $a$ for $|a|$ sufficiently small and
since $c$ is even in $a$, we write that
$$
w(k, a)(z):= a \cos z+a^2 w_2(z)+a^3 w_3(z)+O\left(a^4\right)
$$
and
$$
c(k, a):= c_0(k)+a^2 c_2+O\left(a^4\right)
$$
as $a\rightarrow0$, where $w_2, w_3,$ . . . are even and $2\pi$-periodic in $z$. Substituting these into \eqref{2.6}, at the order of $a^2$, we gather that
\[
\left(\kappa -c_0\right) w_2+\frac{\kappa}{1+k^2} k^2 \partial_z^2 w_2+\frac{(b+1)}{2} \cos ^2z+k^2 \cos^2 z-\frac{(b-1)}{2}k^2\sin ^2z=0,
\]
which is equivalent to
\begin{align}
\frac{ \kappa k^2}{1+k^2} w_2+\frac{\kappa k^2}{1+k^2} \partial^2_z w_2&=\frac{(b-1)}{2} k^2\sin ^2 z-\left(\frac{(b+1)}{2}+k^2\right) \cos ^2 z\nonumber\\
&=\frac{(b-1)}{4}k^2(1-\cos 2 z)-\left(\frac{(b+1)+2k^2}{4}\right)\left(\cos 2z+1\right).\nonumber
\end{align}
Then we get that
\begin{equation}\label{2.16}
\frac{\kappa k^2}{1+k^2}( w_2+\partial^2_z w_2)
=-\frac{(b+1)+(3-b)k^2}{4}-\frac{(b+1)(1+k^2)}{4}\cos 2z.
\end{equation}
A straightforward calculation then reveals that
\begin{equation}\label{2.17}
w_2=\frac{\left(1+k^2\right)}{4\kappa k^2}\left[\frac{(b+1)(k^2+1)}{3} \cos 2z+\left((b-3)k^2+(b+1)\right)\right].
\end{equation}
At the order of $a^3$,
\[
\begin{aligned}
& \left(\kappa -c_0\right) w_3-c_2 \cos z+c_0 k^2 \partial_z^2 w_3-c_2 k^2 \cos z+{(b+1)}\cos z w_2 \\
& -k^2\left(\cos z \partial_z^2 w_2-w_2 \cos z\right)+{(b-1)}k^2 \sin z \partial_z w_2=0,
\end{aligned}
\]
which is equivalent to
\begin{equation}\label{2.18}
\begin{aligned}
\frac{\kappa k^2}{1+k^2}\left(w_3+\partial_z^2 w_3\right)= & -{(b+1)} \cos z w_2+c_2\left(1+ k^2\right) \cos z\\
%\frac{1 + k^2}{8 m k^2}\left[\left(k^2+1\right) \cos 2 z-\left(k^2+3\right)\right] \\
& +k^2\cos z \left(\partial_z^2 w_2-w_2 \right)-{(b-1)}k^2 \sin z \partial_z w_2.
\end{aligned}
\end{equation}
From \eqref{2.17}, we get that
\[
\partial_z^2w_2-w_2=\frac{1+k^2}{4 \kappa k^2}\left[-\frac{5(b+1)}{3}\left(k^2+1\right)\cos2z
+\left((3-b)k^2+(b+1)\right)\right],
\]
which helps us to get that
\[
\begin{aligned}
&k^2 \cos z\left(\partial_z^2 w_2-w_2\right)-(b-1)k^2 \sin z \p_z w_2\\
&=k^2 \frac{1 + k^2}{2\kappa k^2}\left[-\frac{5(b+1)}{6}\left(k^2+1\right) \cos 2z \cos z+\left(\frac{(3-b)k^2+(b+1)}{2}\right) \cos z \right.\\
&\quad\left.+ \frac{b^2-1}{3}\left(k^2+1\right) \sin z \sin 2 z\right]\\
&=\frac{1+k^2}{2\kappa}\left[\frac{b^2-1}{3}\left(k^2+1\right)\cos (2 z-z)-\frac{(b+1)(2b+1)}{6}\left(k^2+1\right) \cos 2 z \cos z\right.\\
&\quad\left.+\left(\frac{(3-b)k^2+(b+1)}{2}\right) \cos z\right]\\
&= \frac{1+k^2}{2 \kappa }[\left(3k^2+5\right)\cos z-7\left(k^2+1\right)\cos 2 z \cos z].
\end{aligned}
\]
Then the right hand side of \eqref{2.18} equals to
\[
\begin{aligned}
&\frac{1+k^2}{2\kappa}  {\left[\frac{k^2(2b^2-3b+7)+(2b+1)(b+1)}{6} \cos z-\frac{(b+1)(2b+3)}{6}\left(k^2+1\right) \cos 2z \cos z \right.} \\
& \qquad\quad\left.-\frac{(b+1)^2}{6k^2}\left(k^2+1\right) \cos 2 z \cos z-\frac{(b+1)\left(b(k^2-1)-(3k^2+1)\right)}{2k^2} \cos z+2 \kappa c_2 \cos z\right]\\
&=-\frac{1+k^2}{2\kappa}\left[\frac{b+1}{6}\left((2b+3)+\frac{(b+1)}{k^2}
\right)\left(k^2+1\right) \cos2z  \cos z\right.\\
&\quad\left.-\left(\frac{k^4(2b^2-3b+7)-k^2(b+1)(b-10)+3(b+1)^2}{6k^2}+2\kappa c_2\right) \cos z\right].
\end{aligned}
\]
Noting that
\[
\cos 3\alpha=2\cos 2\alpha \cos\alpha-\cos\alpha,
\]
we choose
\begin{equation}\label{2.19}
c_2=\frac{1}{\kappa}\left(\frac {-2b^2+11b-11}{24}k^2 +\frac {5b^2-11b-16}{24}-\frac{5(b+1)^2}{24k^2}\right).
\end{equation}
Then \eqref{2.18} gives
\[
\frac{\kappa k^2}{1+k^2}\left(w_3+\partial_z^2 w_3\right)= -\frac{1+k^2}{\kappa}
\frac {\left(k^2+1\right)(b+1)}{24}\left((2b+3)+\frac{b+1}{k^2}\right)\cos 3z,
\]
which is equivalent to
\[
\left(w_3+\partial_z^2 w_3\right)= -\frac{(b+1)\left(k^2+1\right)^3\left((2b+3) k^2+(b+1)\right)}{24 \kappa^2 k^4}\cos 3z.
\]
A straightforward calculation then reveals that
\begin{equation}\label{2.20}
w_3= \frac{(b+1)\left(k^2+1\right)^3}{192 \kappa^2 k^4}\left((2b+3) k^2+(b+1)\right)\cos 3z.
\end{equation}

\section{Proof of Lemma \ref{lem4.4} and Lemma \ref{lem5.7}}\label{app B}
\begin{proof}[Proof of Lemma \ref{lem4.4}]Note that
{ {
\[
\ell_a^2=
\left((b+1)+(7-2b)k^2\right)\frac{(b+1)\left(1+k^2\right)^2}{12\kappa k^2}a^2.
\]
}}
Now we discuss $b$ in detail.

(1) $(b+1)>0$

(1a) If $b\leq\frac 7 2$, i.e. $-1< b\leq\frac 7 2$, we have $\ell_a^2>0$.

(1b) If $b>\frac 7 2$, we have $\ell_a^2>0$ for $k^2< \frac {b+1} {2b-7}$ and $\ell_a^2\leq0$ for
$k^2\geq \frac {b+1} {2b-7}$.

(2) $(b+1)<0$

(2a) If $k^2< \frac{b+1} {2b-7}$, we have $\ell_a^2>0$.

(2b) If $k^2\geq \frac {b+1} {2b-7}$, we have $\ell_a^2 \leq0$.

(3) $(b+1)=0$, $\ell_a^2=0$.
\end{proof}
\begin{proof}[Proof of Lemma \ref{lem5.7}] We can rewrite $B$ as
\[
\begin{aligned}
B=&\left[ k^2\xi^2+(1-b)k^2\xi+k^2+(b+1)\right]
\left[k^2\xi^2+(b-3)k^2\xi+(3-b)k^2+(b+1)\right] \\
=&\left[ k^2\left(\xi+\frac {(1-b)}{2}\right)^2+\left(1+b\right)\left(\frac {k^2}{4}(3-b)+1\right)\right] \\
&\ \cdot \left[ k^2\left(\xi+\frac {(b-3)}{2}\right)^2+\left(1+b\right)\left(\frac {k^2}{4}(3-b)+1\right)\right].
\end{aligned}
\]

(I) In the case of $-1\leq b\leq 3$, $\left(1+b\right)\left(\frac {k^2}{4}(3-b)+1\right)\geq0$ and thus $B>0$.

(II) In the case of $b<-1$, we have $\left(1+b\right)\left(\frac {k^2}{4}(3-b)+1\right)<0$. Then

(II-1) for $\xi=\frac 1 2$, we have $B=\left(\frac {(7-2b)} {4} k^2+(1+b)\right)^2$, then

(II-1-i) for $k^2=\frac {-4(1+b)}{7-2b}$, we have $B=0$.

(II-1-ii) for $k^2\neq \frac {-4(1+b)}{7-2b}$, we have $B>0$.

(II-2) for $\xi\in(0,\frac 1 2)$, we have $|\xi+\frac {(1-b)}{2}|<|\xi+\frac {(b-3)}{2}|$, and thus

(II-2-i) for $k^2>\frac {-\left(1+b\right)\left(\frac {k^2}{4}(3-b)+1\right)} {\left(\xi+\frac {(1-b)}{2}\right)^2}$ or $k^2<\frac {-\left(1+b\right)\left(\frac {k^2}{4}(3-b)+1\right)} {\left(\xi+\frac {(b-3)}{2}\right)^2}$, i.e. $k^2>\frac {-(1+b)}{\xi^2+(1-b)\xi+1}$ or $k^2<\frac {-(1+b)}{\xi^2+(b-3)\xi+(3-b)}$ , we have  $B>0$.

(II-2-ii)  for $\frac {-(1+b)}{\xi^2+(1-b)\xi+1} < k^2 < \frac {-(1+b)}{\xi^2+(b-3)\xi+(3-b)}$, we have  $B < 0$.

\medskip

(III) In the case of $b>3$, we have

(III-1) for $k^2\leq\frac {4} {b-3}$, we have $\left(1+b\right)\left(\frac {k^2}{4}(3-b)+1\right)\geq0$ and thus $B>0$.

(III-2) for $k^2> \frac {4} {b-3}$, we have $\left(1+b\right)\left(\frac {k^2}{4}(3-b)+1\right)<0$, and thus we can proceed a similar discussion as in (II).

(III-2-i) for $\xi=\frac 1 2$, we have $B=\left(\frac {(7-2b)} {4} k^2+(1+b)\right)^2$, then

(III-2-i-a) for $3<b\leq \frac 7 2$, we have $B>0$.

(III-2-i-b) for $b> \frac 7 2$, we have $B=0$ for $k^2=\frac {-4(1+b)}{7-2b}>\frac {4} {b-3}$ and $B>0$ for $k^2>\frac {4} {b-3}$ and $k^2\neq\frac {-4(1+b)}{7-2b}$.

(III-2-ii) for $\xi\in(0,\frac 1 2)$, we have $\frac {\xi^2+\xi+1} {\xi}>\frac {\xi^2-3\xi+3} {1-\xi}>3$, then

(III-2-ii-a) for $3<b\leq \frac {\xi^2-3\xi+3} {1-\xi}$, we have
\[
0\leq\left(\xi^2+(b-3)\xi+(3-b)\right)<\left(\xi^2+(1-b)\xi+1\right),
\]
then $B>0$.

(III-2-ii-b) for $\frac {\xi^2-3\xi+3} {1-\xi}<b\leq\frac {\xi^2+\xi+1} {\xi}$, we have
\[
\left(\xi^2+(b-3)\xi+(3-b)\right)<0\leq \left(\xi^2+(1-b)\xi+1\right),
\]
then we have $B>0$ for $\frac {4} {b-3}<k^2<\frac {-(1+b)}{\xi^2+(b-3)\xi+(3-b)}$ and $B < 0$ for $k^2 > \frac {-(1+b)}{\xi^2+(b-3)\xi+(3-b)}$.

(III-2-ii-c) for $b>\frac {\xi^2+\xi+1} {\xi}$, we have
\[
\left(\xi^2+(b-3)\xi+(3-b)\right)<\left(\xi^2+(1-b)\xi+1\right)<0,
\]
then we have $B>0$ for $k^2>\frac {-(1+b)}{\left(\xi^2+(1-b)\xi+1\right)}$ or $\frac {4} {b-3}<k^2<\frac{-(1+b)}{\xi^2+(b-3)\xi+(3-b)}$, and $B < 0$ for $\frac {-(1+b)}{\xi^2+(b-3)\xi+(3-b)} < k^2 < \frac{-(1+b)}{\left(\xi^2+(1-b)\xi+1\right)}$.
\end{proof}
\vspace{0.5cm}
\noindent {\bf Acknowledgements}
The work of LF is partially supported by a NSF of Henan Province of China Grant No. 222300420478, the research of RMC is supported in part by the NSF through DMS-2205910, the research of XCW is supported by the National Natural Science Foundation of China No. 12301271, and the research of RZX is supported by the National Natural Science Foundation of China No. 12271122.

\vspace{0.5cm}
\noindent {\bf Conflict of interest}
The authors declare that there is no conflict of interest.

\vspace{0.5cm}
\noindent {\bf Data Availability}
There is no data associated to this work.

\bibliographystyle{siam}
\bibliography{bKP}

\end{document}